\documentclass[11pt]{article}
\usepackage{amsmath, amsfonts, amsthm, amssymb,sectsty,dsfont}

\usepackage[english]{babel}

\usepackage{graphicx}
\usepackage{fullpage}
	\usepackage{color}
\usepackage{xcolor}
\usepackage{bbm}

\usepackage{subcaption}

\usepackage{algorithm}
\usepackage{algpseudocode}

 \usepackage{mathrsfs}
 \usepackage{mathtools}
 \mathtoolsset{showonlyrefs}

\title{Asymptotic properties of parallel Bayesian kernel density estimators}

\author{Alexey Miroshnikov \thanks {Department of Mathematics, University of California, Los Angeles, amiroshn@gmail.com} \and Evgeny Savelev\thanks{Department of Mathematics, Virginia Polytechnic Institute and State University, savelev@vt.edu}}

\newcommand{\diag}{{\text{diag}}}

\newcommand{\RR}{\mathbb{R}}
\newcommand{\eps}{\varepsilon}

\def\clap#1{\hbox to 0pt{\hss#1\hss}}

\newcommand{\SP}{\hspace{1pt}}

\newtheorem{theorem}{Theorem}[section]
\newtheorem*{proposition*}{Proposition}
\newtheorem{proposition}[theorem]{Proposition}
\newtheorem*{theorem*}{Theorem}

\newtheorem*{convention*}{Convention}
\newtheorem*{corollary*}{Corollary}
\newtheorem*{definition*}{Definition}
\newtheorem{definition}[theorem]{Definition}
\newtheorem{lemma}[theorem]{Lemma}
\newtheorem*{lemma*}{Lemma}
\newtheorem{remark}[theorem]{Remark}
\newtheorem*{remark*}{Remark}
\newtheorem*{mtheorem*}{Main Theorem}

\numberwithin{equation}{section}

\sectionfont{\large}
\subsectionfont{\normalsize}

\newcommand{\bias}{\text{\rm bias}}

\newcommand{\MISE}{\mbox{\rm MISE}}
\newcommand{\MISEB}{\overline{\MISE}}

\newcommand{\Cov}{\mbox{\rm Cov}}
\newcommand{\AMISE}{\mbox{\rm AMISE}}
\newcommand{\AMISEB}{\overline{\AMISE}}

\newcommand{\EXP}{\mathbb{E}}
\newcommand{\VAR}{\mathbb{V}}

\newcommand{\fhat}{\widehat{f}}
\newcommand{\phat}{\widehat{p}}
\newcommand{\lhat}{\widehat{\lambda}}
\newcommand{\chat}{\widehat{c}}

\newcommand{\PP}{\mathbb{P}}

\newcommand{\HOPT}{h^{\rm opt}}
\newcommand{\hopt}{\HOPT}
\newcommand{\hoptv}{{\bf h}^{\rm opt}}

\newcommand{\HOPTAM}{\hoptv}
\newcommand{\NN}{\mathbf{N}}
\newcommand{\hh}{\mathbf{h}}

\newcommand{\ds}{\displaystyle}

\def\wh{\widehat}

\date{}

\begin{document}

\maketitle


\abstract{

In this article we perform an asymptotic analysis of Bayesian parallel kernel density estimators introduced by Neiswanger, Wang and Xing \cite{NCX14}. We derive the asymptotic expansion of the mean integrated squared error for the full data posterior estimator and investigate the properties of asymptotically optimal bandwidth parameters. 
Our analysis demonstrates that partitioning data into subsets requires a non-trivial choice of bandwidth parameters that optimizes the estimation error.
}

\tableofcontents

\section{Introduction}

Recent developments in data science and analytics research have produced an abundance of large data sets that are too large to be analyzed in their entirety. As the size of data sets increases, the time required for processing rises significantly. An effective solution to this problem is to perform statistical analysis of large data sets with the use of parallel computing. The prevalence of parallel processing of large data sets motivated a surge in research on parallel statistical algorithms.

One approach is to divide data sets into smaller subsets, and analyze the subsets on separate machines using parallel Markov chain Monte Carlo (MCMC) methods \cite{LSZ2009,NASW2009, SN2010}. These methods, however, require communication between machines for generation of each sample. Communication costs in modern computer networks dwarf the speed up achieved by parallel processing and therefore algorithms that require extensive communications between machined are ineffective; see Scott \cite{S2016}. 


To address these issues, numerous alternative communication-free parallel MCMC methods have been developed for Bayesian analysis of big data. These methods partition data into subsets, perform independent Bayesian MCMC analysis on each subset, and combine the subset posterior samples to estimate the full data posterior; see \cite{SBBB2014, NCX14, MWC15}.

Neiswanger, Wang and Xing \cite{NCX14} introduced a parallel kernel density estimator that first approximates each subset posterior density; the full data posterior is then estimated by multiplying the subset posterior estimators  together,
\begin{equation}\label{estmod}
   {\wh p}({\bf x}|{\bf y}) \; \propto \;
 \ds \phat^*({\bf x}|{\bf y}) := {\wh p}_1({\bf x}|{\bf y}_1) \cdot {\wh p}_2({\bf x}|{\bf y}_2) \dots \cdot {\wh p}_M({\bf x}|{\bf y}_M)\,.
\end{equation}
Here ${\bf x} \in \RR^d$  is the model parameter,  ${\bf y} = \{ {\bf y}_1, {\bf y}_2, \dots, {\bf y}_M \}$ is the full data set partitioned into $M$ disjoint independent subsets, and 
\begin{equation}\label{KDE}
{\wh p}_m({\bf x}|{\bf y}_m) = \sum_{i=1}^{N_m}\frac{1}{h_m}K\Big(\frac{{\bf x}-{\bf X}_i^m}{h_m}\Big)
\end{equation}
is the subset posterior kernel density estimator, with $h_m \in \RR_+$ a kernel bandwidth parameter.


The authors of \cite{NCX14}  show that the estimator \eqref{estmod} is asymptotically exact and develop a sampling algorithm that generates samples from the distribution approximating the full data estimator. Similar  sampling algorithms were presented and investigated in Dunson and Wang \cite{WD14} and Scott \cite{SBBB2014,S2016}. It has been noted that these algorithms do not perform well for posteriors that have non-Gaussian shape and are sensitive to the choice of the kernel parameters; see \cite{MWC15,SBBB2014,WD14}.

 The highlighted issues indicate that the proper choice of the bandwidth can greatly benefit the accuracy of the estimation as well as sampling algorithms. Moreover, properly chosen bandwidth parameters will improve accuracy of the estimation without incurring additional computational cost.

In the present article, we are concerned with an asymptotic analysis of the parallel Bayesian kernel density estimators of the form \eqref{estmod}. In particular, we are interested in the asymptotic representation of the mean integrated squared error (MISE) for the non-normalized estimator $\phat^*$ and the density estimator $\phat$ as well as the properties of the optimal kernel bandwidth vector parameter $\hh=(h_m)_{m=1}^M$ as $\NN=(N_1,N_2,\dots,N_M) \to \infty$; the issues left open in \cite{NCX14}. 

\par

We also propose a universal iterative algorithm  based on the derived asymptotic expansions that locates optimal parameters without adopting any assumptions on the underlying probability densities.

The  kernel density estimators for the case $M=1$ have been studied extensively in the past five decades. Asymptotic properties of the mean integrated squared error for the estimator \eqref{estmod} with $M=1$ and $d=1$, which takes the form \eqref{KDE}, were studied by Rosenblatt~\cite{RZ56}, Parzen~\cite{P1962} and Epanechnikov~\cite{EP69}. In particular, for sufficiently smooth probability densities Parzen~\cite{P1962}  derived the asymptotic expansion for the mean integrated squared error 
\begin{equation}\label{intromiseexact}
\MISE[p,\phat,\NN,\hh]= \frac{h^4k^2_2}{4}\int_{\RR}(p''(x))^2 dx
+ \frac{1}{nh}\int_{\RR}K^2(t)\,dt+
o\Big( \frac{1}{n h} + h^4\Big)\,,
\end{equation}
with $\NN=n$ and $\hh=h$, and obtained a formula for the asymptotically optimal bandwidth parameter 
\begin{equation}\label{introhoptimal1d}
\HOPT_{M=1}=n^{-1/5}k_2^{-2/5}\left(\int_{\RR}K^2(t)\,dt\right)^{1/5}\Big(\int_{\RR}\big(p''(x)\big)^2 \, dx\Big)^{-1/5}\,
\end{equation}
which minimizes the leading terms in the expansion.

The case of non-differentiable or discontinuous probability density functions has been shown to possess different asymptotic estimates for $\MISE$. It has been shown by van~Eden~\cite{EED85} that the optimal bandwidth parameter $h_{M=1}^{\rm opt} \in \RR$ and the rate of convergence of the mean integrated squared error depend directly on the regularity of the probability density $p$.

In the case of multivariate distributions, $d\geq 1$, the complexity of the asymptotic analysis depends on the form of the bandwidth matrix ${\bf H} \in \RR^{d \times d}$. In the simplest case, one can assume that ${\bf H}=h {\bf I}$, where $h$ is a scalar; see Silverman \cite{SILV85}, Simonoff \cite{SIM1996} and Epanechnikov \cite{EP69}. Another approach is to consider the bandwidth matrix of the form ${\bf H}=\diag(h_1,h_2,\dots,h_d)$, with $h_i$ being a bandwidth parameter for each dimension $i \in \{1,\dots,d\}$. The most general formulation assumes that $\mathbf{H}$ is a $d{\times}d$ matrix, which allows one to encode correlations between components of $\mathbf{x}$; see Duong and Hazelton\cite{DH2005}, and Wand and Jones \cite{WJ1994}.

In the present work, motivated by the ideas of \cite{P1962,DH2005,WJ1994,RZ56} we focus on the case $M>1$ and $d=1$ and do the asymptotic analysis of the mean integrated squared error for both the parallel non-normalized estimator
\begin{equation*}
\begin{aligned} \textstyle
\MISE \big[\phat^* ,p^* \, ; \NN,\hh \big] & = \EXP \int_{\RR} \Big\{ p^*({x}|{\bf y})-\phat^*({x}|{\bf y})) \Big\}^2 dx
 \end{aligned}
\end{equation*}
and the full data set posterior density estimator
\begin{equation*}
\begin{aligned} \textstyle
\MISE \big[\phat ,p \, ; \NN,\hh \big] & = \EXP \int_{\RR} \Big\{ p({x}|{\bf y})-{\wh p}({x}|{\bf y})) \Big\}^2 dx
 \end{aligned}
\end{equation*}
as 
\[
\NN=(N_1,N_2,\dots,N_M) \to \infty, \quad {\bf h}=(h_1,h_2,\dots,h_M) \to 0 \quad \text{and} \quad (\NN \cdot \hh)^{-1} \to 0.
\]

\par

In Theorem \ref{multimiseest}, under appropriate condition on the regularity of the probability density, we  derive  the expression for $\AMISE[p^*,\phat^*]$, the asymptotically leading part of $\MISE$ for the estimator $\phat^*$. 
The leading part turns out to be in agreement with the leading part for the case $M=1$, but in the multi-subset case, $M>1$, the leading part contains novel terms that take into account the relationship between $M$ subset posterior densities $p_m$.


We then perform a similar analysis for the mean square error of the full data set posterior density estimator $\phat$. The presence of the normalizing constant 
\[
\chat =\widehat{\lambda}^{-1}= \bigg(\int \phat_1(x|{\bf y}) \cdot \phat_2(x|{\bf y})\dots \phat_M(x|{\bf y}) dx \bigg)^{-1} = \bigg(\int \phat^*(x|{\bf y}) \, dx \bigg)^{-1}\,
\]
introduces major difficulties in the analysis of MISE because $\chat$ may in general have an infinite second moment in which case $\MISE[\phat,p]$ is not defined. This may occur when the estimators $\phat^*_i$ (on some events) decay too quickly in $x$ variable and the sets of $x$ with the most `mass' for each $\phat^*_i$ have little common  intersection, which potentially leads to large values of $\chat$. To make sure that $\EXP\chat^2<\infty$ one must impose appropriate conditions on the density $p$ and kernel $K$. In this article, however, we take another approach. Instead, we replace the mean integrated squared error by an asymptotically equivalent distance functional denoted by
\begin{equation*}
\begin{aligned} \textstyle
\overline{\MISE} \big[\phat, p \, ; \NN,\hh \big] & = \EXP \bigg[\left(\frac{\lhat}{\lambda}\right)^2\int_{\RR} \Big\{ p({x}|{\bf y})-{\wh p}({x}|{\bf y})) \Big\}^2 dx \bigg].
 \end{aligned}
\end{equation*}
We show that the new functional is always well-defined and that it is asymptotically equivalent to $\MISE$ when restricted to events $\Omega_{\NN} \subset \Omega$ whose probability tends to one as the total number of samples $\|\NN\| \to \infty$.

We then do the analysis of the functional $\MISEB$ by carrying out the same program as for the $\MISE$ of the estimator $\phat^*$. In Theorem \ref{AMISEplmm} we derive the expression for $\AMISEB[p,\phat]$, the asymptotically leading part of the $\MISEB$ for the full data set posterior density estimator $\phat$. The asymptotically optimal bandwidth parameter for the full data set posterior is then defined to be a minimizer
 \[
\HOPTAM = \mathrm{argmin}_{\hh\in\RR^M_{+}}\AMISEB[p, \phat; \NN,\hh]\,.	
 \]

We then compute minimizing bandwidth $\HOPTAM$ in explicit form for two special cases. In the examples presented here we consider subset posterior densities of normal and gamma distributions; see \eqref{hoptsym}, \eqref{hoptbandnorm}, and \eqref{hoptgamma}. In the two examples the optimizing bandwidth vectors differ significantly and depend, as expected, directly on the full data set density which is typically unknown. For that reason we propose an iterative algorithm for locating optimal bandwidth parameters based on asymptotic expansion we derived; see Algorithm \ref{minhalg}. 



Our analysis demonstrates that partitioning data into $M>1$ sets affects the optimality condition of parameter ${\hh}$. It also indicates that the bandwidth vector 
\[
\hh^{\rm opt}_0=(\hopt_{1,M=1},\hopt_{2,M=1},\dots,\hopt_{M,M=1})\,,
\] 
which minimizes the `componentwise' mean integrated squared error
\[
\sum_{i=1}^M \MISE[\phat_i,p_i,N_i,h_i]\,,  
\]
where $\HOPT_{m,M=1}$ is the optimal bandwidth parameter for the estimator  $\phat_m(x|{\bf y_m})$ given by \eqref{introhoptimal1d}, is suboptimal for both estimators $\phat^*$ and $\phat$ whenever $M>1$. 


This observation highlights the fact that the choice of optimal parameters for parallel kernel density estimators (suitable for parallelizing data analysis) must differ from the theoretical choice suggested in case of processing on a single machine. We must also note, that the increased values of $\MISE$ resulted from choosing a suboptimal bandwidth parameter get compounded in case of parallel processing. This further necessitates the importance of a proper choice of bandwidth, especially if it comes at no additional computational costs.

The paper is arranged as follows. In Section \ref{sec::notation} we set notation and hypotheses that form the foundation of the analysis. In Section \ref{sec_mise_nonnorm} we derive an asymptotic expansion for $\MISE$ of the non-normalized estimator as well as derive formulas for leading parts of $\bias[\phat^*]$ and $\VAR[\phat^*]$, which are central to the analysis performed in subsequent sections. In Section \ref{sec_mise_norm} we perform the analysis of $\MISEB$ for the full data set posterior density. In Section \ref{examplesSec} we compute explicit expressions for optimal bandwidth parameters for several special cases and conduct numerical experiments.  Finally, in the appendix we provide supplementary lemmas and theorems employed in Section \ref{sec_mise_nonnorm} and Section \ref{sec_mise_norm}.

\section{Notation and hypotheses}\label{sec::notation}

For the convenience of the reader we collect in this section all hypotheses and results relevant to our analysis and present the notation that is utilized throughout the article.
\par
\begin{enumerate}
\renewcommand{\labelenumi}{\textbf{(\theenumi)}}
\renewcommand{\theenumi}{H\arabic{enumi}}
	\item\label{model} Motivated by the form of the posterior density at Neiswanger et al. \cite{NCX14} we consider the probability density function of the form
\begin{equation}
p(x) \propto p^*(x) \quad \text{where} \quad  p^*(x):=\prod_{m=1}^M p_m(x)
\end{equation}
Here $p_m(x)$ is a probability density function for each $m \in \{1,\dots,M \}$ . 
	\item\label{hyp2} We consider the estimator of $p$ in the form
	\begin{equation}\label{estmodel}\tag{H2-a}
	   \hat{p}(x)\propto
		\ds \hat{p}^*(x) \quad \text{where} \quad \hat{p}^*(x):=\prod_{m=1}^{M}\hat{p}_m(x)
	\end{equation}
	and for each $m \in \{1,\dots,M\}$  $\hat{p}_m(x)$ is the kernel density estimator of the probability density $p_m(x)$ that has the form
	\begin{equation}\label{subpost}\tag{H2-b}
	\hat{p}_m(x)=\frac{1}{N_m h_m}\sum_{i=1}^{N_m}K\left(\frac{x-X^{m}_{i}}{h_m}\right).	
	\end{equation}
	Here $X_1^m, X_2^m, \dots,X^m_{N_m} \sim p_m(x)$ are independent identically distributed random variables, $K$ is a kernel density function, and $h_m>0$ is a bandwidth parameter.
\end{enumerate}
\par

The mean integrated squared error of the estimator $\phat^*$ of the non-normalized product $p^*$ as well as for the estimator $\phat(x)$ of the full posterior density $p(x)$ is defined by
\begin{equation}\label{mise}
	\begin{aligned}
	\MISE[p^*,\phat^*,\NN,\hh] &= \MISE[p^*, \phat^*(x)] :=  \EXP \int_{\RR} (\phat^*(x) - p^*(x))^2 \, dx\\
	\MISE[p,\phat,\NN,\hh] &= \MISE[p, \phat(x)] :=  \EXP \int_{\RR} (\phat(x) - p(x))^2 \, dx\\
	\end{aligned}
\end{equation}
where we use the notation $\hh = (h_m)_{m=1}^M$ and $\NN = (N_m)_{m=1}^N$. We also use the following convention for the bias and variance of estimators $\phat(x), \phat^*(x), \phat_m(x)$
\begin{equation}\label{biasvarconv}
\begin{aligned}
    \bias[\hat{p}(x)] &= \EXP \big[\hat{p}(x)\big] - p(x)  \\
    \bias[\hat{p}^*(x)] &= \EXP \big[\hat{p}^*(x)\big] - p^*(x)  \\
    \bias[\hat{p}_m(x)] &= \EXP \big[\hat{p}_m(x)\big] - p_m(x)\,, \quad m \in \{1,\dots,M\}.
\end{aligned}
\end{equation}


We assume that the kernel density function $K$ and probability densities functions $p_1,\dots,p_M$ satisfy the following hypotheses:
\begin{enumerate}
\setcounter{enumi}{2}
\renewcommand{\labelenumi}{\textbf{(\theenumi)}}
\renewcommand{\theenumi}{H\arabic{enumi}}
  \item\label{kcond1}  $K$ is positive, bounded, normalized, and its first moment is zero, that is
  \begin{equation}
0 \leq K(t) \leq C, \quad \int_{\mathbb{R}}K(t)\,dt=1\,,\quad
\int_{\mathbb{R}}t\,K(t)\,dt=0\,, \quad \int_{\mathbb{R}} K^2(t)\,dt < \infty
\end{equation}

  \item\label{kcond2} For each  $s \in \{0,1,2,3\}$
\begin{equation}
k_s=\int_{\mathbb{R}}|t|^{s} K(t)\,dt \, < \, \infty\,.
\end{equation}

\item\label{pcond1} For each $m \in \{1,\dots,M\}$, $s \in \{0,1,2,3\}$ and density $p_m \in C^3(\RR)$ there exists a constant $C \geq 0$ such that
\begin{equation}
|p_m^{(s)}(x)|< C \quad\text{for all } x \in \RR\,.
\end{equation}

\item\label{pcond2} For each $m \in \{1,\dots,M\}$ and $s \in \{0,1,2,3\}$ the density $p_m(x)$ and its derivatives are integrable, that is, there is a constant $C$ so that 
\begin{equation}
	\int_{\RR}|p_m^{(s)} (x)|\,dx=C<\infty\,.
\end{equation}

%

\item\label{unifconvNh} Functions
\[
\begin{aligned}
 \NN(n)&=\{N_1(n),N_2(n),N_3(n),\ldots,N_M(n)\}:\mathbb{N}\to\mathbb{N}^M\\
 \hh(n)&=\{h_1(n),h_2(n),\ldots,h_M(n)\}:\mathbb{N}\to\RR_{++}^M
\end{aligned}
\]
satisfy for all $i \in \{1,2,\ldots, M\}$
\begin{equation}
\begin{aligned}
&D_1\leq \frac{N_i}{n}\leq D_2 \quad \text{for some }0<D_1<D_2 \\
&  A_1 N_i(n)^{-\alpha_0} \leq h_i(n) \leq A_2 N_i(n)^{-\alpha_0} \quad \text{for some} \quad \alpha_0\in(0,1)\\[2pt]
&\lim_{n\to\infty}{h_i(n) N_i(n)}=\infty\,.\\
\end{aligned}
\end{equation}
We also define $\underline{N}(n)=\min_{i} N_i(n)$ and note that $C_1\|\NN\| \leq  \underline{N}(n) \leq C_2\|\NN(n)\| $.





\end{enumerate}

\par\smallskip


\section{Asymptotic analysis of $\MISE$ for $\phat^*$}\label{sec_mise_nonnorm}


We start with the observation that $\MISE$ can be expressed via the combination of bias and variance

\begin{equation}\label{multimiseprod}
  \begin{aligned}
       \MISE[p^*,\phat^*]&=\EXP \int_{\mathbb{R}}  \big(\phat^*(x)-p^*(x)\big)^2\,dx\\
     &  = \int_{\RR} \Big(\bias\big[\phat^*(x)\big] \Big)^2\, dx + \int_{\RR} \VAR[\phat^*(x)] \, dx.
  \end{aligned}
\end{equation}

In what follows we do the analysis of the bias, then that of variance and conclude with the section where we derive the formula for the optimal bandwidth vector.

\subsection{Bias expansion}

Using the fact that $\phat_i(x)$, $i=1,\dots,M$ are independent, we obtain
\begin{equation}\label{multiprodbias}
\begin{aligned}
	\bias[\phat^*(x)]&=\EXP[\phat^*(x)]-p^*(x) \\[3pt]
    & =\prod_{m=1}^M \EXP[\phat_m](x)-\prod_{m=1}^M p_m(x)\\[3pt]
    & = \prod_{m=1}^M\big(\bias[p_m(x)]+p_m(x)\big)-\prod_{m=1}^M p_m(x)
\end{aligned}
\end{equation}
To simplify notation in \eqref{multiprodbias} we shall employ the multiindex notation. Let $\alpha$ be the multiindex with
$$
	\alpha = (\alpha_1,\alpha_2,\ldots,\alpha_M)\qquad\alpha_m\in\{0,1\}.
$$
Then the above formula can rewritten as follows
\begin{equation}\label{multiprodbias2}
\begin{aligned}
	\bias[\phat^*(x)]&=\sum_{1\leqslant|\alpha|\leqslant M}\prod_{m=1}^M\bias^{\alpha_m}[\phat_m(x)]\big(p_m(x)\big)^{(1-\alpha_m)}\\
	&=\sum_{m=1}^{M}\left[\bias[\phat_m(x)]\prod_{\substack{k=1\\k \neq m}}^M p_k(x)\right] \\
	&\qquad+ \sum_{2\leqslant|\alpha|\leqslant M}\prod_{m=1}^M \big(\bias[\phat_m(x)]\big)^{\alpha_m}(p_m(x))^{(1-\alpha_m)}.
\end{aligned}
\end{equation}

\par

Using this decomposition, we prove the following lemma
\begin{lemma}\label{miltibiasformula}
Suppose hypotheses \eqref{kcond1}-\eqref{pcond2} hold. Then
\begin{itemize}
\item [$(i)$] The bias can be expressed as
  \begin{equation}\label{miltipiecebiasexp}
  \begin{aligned}
\bias[\phat^*(x)] &= \frac{k_2}{2} \sum_{m=1}^{M}\left[h_m^2 p_m''(x)\prod_{\substack{k=1\\k\neq m}}^M p_k(x)\right] + E_b(x \SP;  \hh)
\end{aligned}
\end{equation}
where the error term $E_b(x; \hh)$ satisfies the bounds
  \begin{equation}\label{multibiaserrbound}
  \begin{aligned}
    |E_b(x \SP; \hh)|  &\leq E_{\infty}||\hh ||^3 \,, \quad \forall x\in \RR \\
    \int_{\RR}  |E_b(x \SP ;\hh )|\, dx  &\leq  E_{1}||\hh||^3 \\
    \int_{\RR} |E_b(x \SP ; \hh)|^2\, dx  & \leq  E_{2}||\hh||^6
    \end{aligned}
\end{equation}
  \item [$(ii)$] The square-integrated bias satisfies
  \begin{equation}\label{multisqintbias}
    \int_{\RR} \bias ^2[\phat^*(x)] \, dx  \, = \frac{k_2^2}{4} \int_{\RR}\left[\sum_{m=1}^{M}h_m^2 p_m''(x)\prod_{\substack{k=1\\k\neq m}}^M p_k(x)\right]^2\,dx + \mathcal{E}_b(\hh) \, < \, \infty
  \end{equation}
  with the error term satisfying
  \begin{equation}\label{multiintbias2err}
  |\mathcal{E}_b(\hh)| \leq C||\hh||^5
  \end{equation}
  where the constant $C$ is independent of $\NN$ and $\hh \in \RR_+^M$.
\end{itemize}
\end{lemma}

\begin{proof}
	According to \eqref{multiprodbias2} and \eqref{biasexp} we have
	\begin{align*}
		&\bias[\phat^*(x)]=\\
		&\qquad=\frac{k_2}{2} \sum_{m=1}^{M}\left[h_m^2 p''_m(x)\prod_{\substack{k=1\\k \neq m}}^M p_k(x)\right]
		+\sum_{m=1}^{M}\left[E_{b,m}\prod_{\substack{k=1\\k \neq m}}^M p_k(x)\right]\\
		&\qquad\qquad+\sum_{2\leqslant|\alpha|\leqslant M}\prod_{m=1}^M \left(\frac{h_m^2 k_2}{2}  p_m''(x)+E_{b,m}\right)^{\alpha_m}(p_m(x))^{(1-\alpha_m)}
	\end{align*}
	Here $E_{b,m}$ is the error in bias approximation for each $\phat_m$ from \eqref{biasexp}. We are computing bounds for
	\begin{equation}\label{multibiaserr}
	\begin{aligned}
	&E_b(x \SP; \hh)=\\
	&\qquad=\sum_{m=1}^{M}\left[E_{b,m}\prod_{\substack{k=1\\k \neq m}}^M p_k(x)\right]+\sum_{2\leqslant|\alpha|\leqslant M}\prod_{m=1}^M \left(\frac{h_m^2 k_2}{2}  p_m''(x)+E_{b,m}\right)^{\alpha_m}(p_m(x))^{(1-\alpha_m)}
	\end{aligned}
	\end{equation}
	To simplify the derivations we separate the terms in \eqref{multibiaserr} into two groups: terms with at least one multiple of $E_{b,m}$ and terms free of $E_{b,m}$. We define the sets
	\begin{equation}\label{amparts}
		A_m =\Big\{\alpha = (\alpha_j)_{j=1}^M\ :\ \alpha_m=0\text{ and }1\leq|\alpha|\leq(M-1)\Big\}
	\end{equation}
	and functions
	\begin{equation}
		P_m(x)=\prod_{\substack{k=1\\k \neq m}}^M p_k(x) + \sum_{\alpha\in A_m}\left[\prod_{\substack{j=1\\j \neq m}}^M \left(\frac{h_j^2 k_2}{2}  p_j''(x) + \mathds{1}_{\{j>m\}}E_{b,j}\right)^{\alpha_j}(p_j(x))^{(1-\alpha_j)}\right].
	\end{equation}
	Here $\mathds{1}$ is the characteristic function. Consequently, the error term can be written as follows
	\begin{equation}\label{rearrngderr}
	\begin{aligned}
	&E_b(x \SP; \hh)=\\
	&\qquad=\sum_{m=1}^{M}\left[E_{b,m} P_m(x)\right]+\sum_{2\leqslant|\alpha|\leqslant M}\prod_{m=1}^M \left(\frac{h_m^2 k_2}{2}  p_m''(x)\right)^{\alpha_m}(p_m(x))^{(1-\alpha_m)}.
	\end{aligned}
	\end{equation}
	Assuming that $||\hh||$ is bounded, \eqref{pcond1} and \eqref{biasexp}, we can conclude that there is a constant $C_P$ so that
	$$
		|P_m(x)|\leq C_P\text{ for any }x\in\RR\text{ and }1\leq m\leq M
	$$
	Using \eqref{pcond1} and \eqref{biasexp}, we conclude that the first term is bounded, and there is a constant $C$ so that
	\begin{equation}\label{biasexpcp}
		\sum_{m=1}^{M}\left|E_{b,m}P_m(x)\right|\leq C\sum_{m=1}^{M}\left(\frac{k_3 h_m^3}{6}\right)\leq CM\frac{||\hh||^{3} k_3}{6}.
	\end{equation}
	The next sum in \eqref{rearrngderr} contains terms are bounded due to \eqref{pcond1}:
	$$
		\left|\frac{h_m^2 k_2}{2}  p_m''(x)\right|\leq\frac{||\hh||^2 C k_2}{2}\qquad\text{and}\qquad |p_m(x)|\leq C
	$$
	For some appropriate constants $C$. Since each one of the products below has at least two terms with $p_m''(x)$ for some $m$, a constant $C_Q$ must exist, so that
	\begin{equation}\label{biasexpcq}
		\left|\sum_{2\leqslant|\alpha|\leqslant M}\prod_{m=1}^M \left(\frac{h_m^2 k_2}{2}  p_m''(x)\right)^{\alpha_m}(p_m(x))^{(1-\alpha_m)}\right|\leq C_Q\frac{||\hh||^4 k^2_2}{4}
	\end{equation}
	The inequalities \eqref{biasexpcp} and \eqref{biasexpcq} imply the first inequality in \eqref{multibiaserrbound}:
	\begin{equation}\label{miltiinfnorm}
		|E_b(x \SP; \hh)|\leq CM\frac{||\hh||^{3} k_3}{6}+\frac{||\hh||^4 k^2_2}{4}C_Q
	\end{equation}

	$L_1$ integrability follows from conditions \eqref{pcond1}, \eqref{pcond2}, the expansion \eqref{rearrngderr} and the second formula in \eqref{biaserrbound}
	\begin{equation}\label{multil1norm}
	\int_{\RR}|E_b(x \SP; \hh)|\,dx\leq C\left(\frac{k_3 ||\hh||^3}{6} + \frac{||\hh||^4 k_2^2 }{4}\right),
	\end{equation}
	which proves the second estimate in \eqref{multibiaserrbound}.\\
	Using the estimates obtained above, we conclude
	$$
	\int_{\RR}|E_b(x \SP; {\bf h})|^2\,dx\leq \sup_{\RR}|E_b(x \SP; {\bf h})|\cdot\int_{\RR}|E_b(x \SP; {\bf h})|\,dx\leq E_{\infty}\cdot E_{1}||\hh||^6
	$$
	Finally, (ii) follows from Cauchy-Schwartz inequality applied to
	$$
	\begin{aligned}
	\bias^2[\phat^*(x)]&=\frac{k_2^2}{4} \left[\sum_{m=1}^{M}h_m^2 p_m''(x)\prod_{\substack{k=1\\k\neq m}}^M p_k(x)\right]^2+\\
											&+E_b(x\SP;\hh) k_2 \left[\sum_{m=1}^{M}h_m^2 p_m''(x)\prod_{\substack{k=1\\k\neq m}}^M p_k(x)\right] +E_b^2(x\SP;\hh)
	\end{aligned}
	$$
	which leads directly to \eqref{multisqintbias} and \eqref{multiintbias2err}
\end{proof}

\subsection{Variance expansion}

We next obtain an asymptotic formula for the variance of $\hat{p}^*$. For the proof of the lemma, we perform the following preliminary calculation
\begin{equation}\label{multivarexpansion}
\begin{aligned}
	\VAR[\phat^*(x)]&=\EXP[(\phat^*(x))^2]-\Big(\EXP[\phat^*(x)]\Big)^2=\prod_{m=1}^{M}\EXP[\phat_m^2]-\prod_{m=1}^{M}\EXP^2[\phat_m]\\
		&=\prod_{m=1}^M\Big(\VAR[\phat_m]+\big(p_m+\bias[\phat_m]\big)^2\Big)-\prod_{m=1}^M \Big(p_m+\bias[\phat_m]\Big)^2\\
		&=\sum_{1\leqslant|\alpha|\leqslant M}\prod_{m=1}^M\big(\VAR[\phat_m]\big)^{\alpha_m}\big(p_m+\bias[\phat_m]\big)^{2(1-\alpha_m)}\\
\end{aligned}
\end{equation}

\begin{lemma}\label{varmulti}
Let hypotheses \eqref{kcond1}-\eqref{unifconvNh} hold. Then
\begin{itemize}
\item [$(i)$] The variation of $\hat{p}^*$ is given by
\begin{equation}\label{multivarexp}
\VAR[\hat{p}^*(x)] =\left(\sum_{m=1}^{M}\left[\frac{p_m}{N_m h_m}\prod_{\substack{k=1\\ k \neq m }}^M p_k^2(x)\right]\right)\int_{\RR}K^2(t)\,dt+ E_{V}(x \SP ; \NN,\hh)\,, \quad x \in \RR
\end{equation}
where the error term $E_V(x \SP; n,\hh)$ satisfies the bounds
\begin{equation}\label{multiL1errvarbnd}
\begin{aligned}
|\mathcal{E}_V(N,h)|&:=\left|\int_{\RR} E_V(x) \,dx \right| = o\left( \frac{1}{\|\NN\|}\right)
\end{aligned}
\end{equation}
\end{itemize}
\end{lemma}
\begin{proof}
According to \eqref{multivarexpansion} we have
\begin{equation}\label{multivarerr}
\begin{aligned}
	&\VAR(\hat{p}^*(x)) = \\
	&\qquad=\sum_{1\leqslant|\alpha|\leqslant M}\prod_{m=1}^M\left(\frac{p_m(x)}{N_m h_m}\int_{\RR}K^2(t)\,dt+E_{V,m}\right)^{\alpha_m}\big(p_m+\bias[\phat_m]\big)^{2(1-\alpha_m)}\\
	&\qquad=\sum_{1\leqslant|\alpha|\leqslant M}\prod_{m=1}^M\left(\frac{p_m(x)}{N_m h_m}\int_{\RR}K^2(t)\,dt+E_{V,m}\right)^{\alpha_m}\big(p_m^2+2p_m\,\bias[\phat_m]+\bias^2[\phat_m]\big)^{(1-\alpha_m)}
\end{aligned}
\end{equation}
Here, $E_{V,m}$ is the approximation error of variance of each $p_m(x)$ from \eqref{varexp}. In a fashion similar to the previous proof, we separate the terms in \eqref{multivarerr}. We single out the leading order terms, the terms with at least one multiple of $E_{V,m}$, the terms with multiples of $\bias[\phat_m]$ and the terms of the order $o\left(\frac{1}{\|\NN\|\|\hh\|}\right)$.

We define sets
\begin{equation}\label{amzparts}
	\begin{aligned}
	A^0_m &=\Big\{\alpha = (\alpha_j)_{j=1}^M\ :\ \alpha_m=0\text{ and }0\leq|\alpha|\leq(M-1)\Big\}\\
	B^1_m &=\Big\{\alpha = (\alpha_j)_{j=1}^M\ :\ \alpha_m=0\text{ and }|\alpha|=1\Big\}
	\end{aligned}
\end{equation}
and functions
\begin{equation}
	\begin{aligned}
	P^0_m(x)&=\sum_{\alpha\in A^0_m}\left[\prod_{\substack{j=1\\j \neq m}}^M \left(\frac{p_m(x)}{N_m h_m}\int_{\RR}K^2(t)\,dt+\mathds{1}_{\{j>m\}}E_{V,m}\right)^{\alpha_m}\big(\EXP^2[\phat_m]\big)^{(1-\alpha_m)}\right],\\
	Q^1_m(x)&=\sum_{\alpha\in B^1_m}\left[\prod_{\substack{j=1\\j \neq m}}^M \left(\frac{p_m(x)}{N_m h_m}\int_{\RR}K^2(t)\,dt \right)^{\alpha_m}\big(\EXP^2[\phat_m]\big)^{(1-\alpha_m)}\right],\\
	\end{aligned}
\end{equation}
The variance expansion can be rewritten as
\begin{equation}
\begin{aligned}
	&\VAR(\hat{p}^*(x)) = \\
	&\qquad=\sum_{1\leqslant|\alpha|\leqslant M}\prod_{m=1}^M\left(\frac{p_m(x)}{N_m h_m}\int_{\RR}K^2(t)\,dt\right)^{\alpha_m}\big(p_m^2+2p_m\,\bias[\phat_m]+\bias^2[\phat_m]\big)^{(1-\alpha_m)}\\
	&\qquad\qquad+\sum_{m=1}^M E_{V,m} P^0_m(x)\\
	&\qquad=\sum_{m=1}^M\left(\frac{p_m(x)}{N_m h_m}\int_{\RR}K^2(t)\,dt\right)\prod_{\substack{j=1\\j \neq m}}^Mp_m^2(x)\\
	&\qquad\qquad+\sum_{m=1}^M\bias[\phat_m](2p_m(x)+\bias[\phat_m])Q_m^1(x)\\
	&\qquad\qquad+\sum_{2\leqslant|\alpha|\leqslant M}\prod_{m=1}^M\left(\frac{p_m(x)}{N_m h_m}\int_{\RR}K^2(t)\,dt\right)^{\alpha_m}\big(\EXP^2[\phat_m]\big)^{(1-\alpha_m)}\\
	&\qquad\qquad+\sum_{m=1}^M E_{V,m} P^0_m(x)\\
\end{aligned}
\end{equation}
Based on definitions of functions $P^0_m(x)$ and $Q^1_m(x)$, hypotheses \eqref{pcond1}, \eqref{pcond2} and \eqref{unifconvNh} we can conclude that there are constants $C_{\EXP}, C_P, C_Q$ so that
$$
	\begin{aligned}
	\EXP[\phat_m]&\leq C_{\EXP}\\
	|P_m^0(x)|&\leq C_P\\
	|Q^1_m(x)|&\leq C_Q \frac{1}{\|\NN\|\|\hh\|}
	\end{aligned}
$$
Therefore
\begin{equation}
	\begin{aligned}
	&\int_{\RR}|E_V(x)|dx\\
	&\qquad\leq \sum_{m=1}^MC\left(2+\frac{||\hh||^2k_2}{2}\right)\frac{C_{Q}}{\|\NN\|\|\hh\|}\int_{\RR}|\bias[\phat_m]|dx\\
	&\qquad\qquad+\frac{1}{\|\NN\|^2\|\hh\|^2}\sum_{2\leqslant|\alpha|\leqslant M}\left(\frac{1}{\|\NN\|^2\|\hh\|^2}\right)^{(|\alpha|-2)}C_{\EXP}^{(M-|\alpha|)}\\
	&\qquad\qquad+\frac{M\cdot C_P}{\|\NN\|}
	\end{aligned}
\end{equation}
This leads directly to \eqref{multiL1errvarbnd}.
\end{proof}

\subsection{$\AMISE$ formula and optimal bandwidth vector}
With the lemmas above we can derive the decomposition of $\MISE[p^*,\phat^*]$ into leading order terms and higher order terms.

\begin{theorem}\label{multimiseest}
Let hypotheses \eqref{kcond1}-\eqref{unifconvNh} hold.  Then $\MISE$ can be represented as
\begin{equation}\label{misedecomp}
\begin{aligned}
	&\MISE[p^*,\phat^*,\NN,\hh]=\AMISE[p^*,\phat^*,\NN,\hh] + \mathcal{E}(\NN,\hh) 
\end{aligned}
\end{equation}
where the leading term 
\begin{equation}\label{multiamise}
	\begin{aligned}
	\AMISE[p^*, \phat^*;\NN,\hh]&=\frac{k_2^2}{4}\int_{\RR}\left(\sum_{m=1}^{M}\left[h_m^2 p_m''(x)\prod_{\substack{k=1\\k\neq m}}^M p_k(x)\right]\right)^2\,dx+\\
	&+ \int_{\RR}\left(\sum_{m=1}^{M}\left[\frac{p_m(x)}{N_m h_m}\prod_{\substack{k=1\\k\neq m}}^M (p_k(x))^2\right]\right)dx\int_{\RR}K^2(t)\,dt
	\end{aligned}
\end{equation}
and the error term $\mathcal{E}$ satisfies
\begin{equation}\label{multiL1errmisebnd}
\mathcal{E}(\NN,\hh) = \mathcal{E}_b(\NN,\hh) + \mathcal{E}_V(\NN,\hh) =o\Big( ||\hh||^4 + \frac{1}{\|\NN\|\|\hh\|}\Big)
\end{equation}
as $\hh \to 0$, $\NN \to \infty$, and $(\|\NN\|\|\hh\|)^{-1} \to 0$.
\end{theorem}
\begin{proof}
The result follows from Lemma \ref{miltibiasformula}, Lemma \ref{varmulti}, and formula \eqref{multimiseprod}
\end{proof}

\begin{remark}\rm
We would like to note that the analysis we perform here is in spirit of the asymptotic analysis performed for multivariate kernel density estimators by Epanechnikov \cite{EP69}. However, the full-data set density  $p$ under consideration is a univariate density expressed as a product and cannot be viewed as a special case of the expansion obtained in \cite{EP69}.
\end{remark}

\begin{remark}\rm
 The asymptotically leading part derived here is the first step of our analysis. It serves  as a stepping stone for the analysis of full $\MISE$  carried out in the next section. We would like to note that one can find optimal bandwidth that minimizes $\AMISE$ for the non-normalized estimator. One has to be aware, however, that these optimal parameters would not take into account a normalization constant and, as a consequence, would be suboptimal for $\MISE$ of the normalized full data set density $\phat$.

\end{remark}

\section{Asymptotic analysis of $\MISE$ for	 $\phat$}\label{sec_mise_norm}

\subsection{Normalizing constant}

In this section we consider the error that arises when one takes into account the normalizing constant. Recall that by assumption
\begin{equation*}
p(x) \propto p^*(x) \quad \text{where} \quad  p^*(x):=\prod_{m=1}^M p_m(x)
\end{equation*}
where $p_m(x)$, $m \in \{1,\dots,M \}$ is a probability density function. Then we define
\begin{equation}\label{normconst}
  \lambda:= \int p^*(x) \, dx > 0 \quad \text{and} \quad c:=\lambda^{-1}
\end{equation}
and obtain $p(x) = c p^*(x)$. For the estimator
\begin{equation*}
   \hat{p}(x)\propto
\ds \hat{p}^*(x) \quad \text{with} \quad \hat{p}^*(x):=\prod_{m=1}^{M}\hat{p}_m(x)
\end{equation*}
we similarly define
\begin{equation}\label{normconstest}
\widehat{\lambda}:= \int \phat^*(x) \, dx > 0 \quad \text{and} \quad \widehat{c}:=\widehat{\lambda}^{-1}
\end{equation}
and hence $\phat(x) = \chat \phat^*(x)$.

\par

We are interested in the optimal bandwidth vector $\hh=(h)_{m=1}^M$ that optimizes the leading term of the mean integrated squared error
\begin{equation}\label{misefull}
\MISE(\hat{p},\, p) = \MISE(\chat\phat^*,\, c p^*)= \EXP \int_{\RR}  \big(c p^*(x) - \chat \phat ^*(x)\big)^2 \, dx\,.
\end{equation}

Observe that $\chat$ and $\phat^*$ are not independent and the previously performed analysis is not directly applicable. Moreover, we observe that the estimator of the normalizing constant
\begin{equation}\label{chatint}
\chat = \bigg(\int \prod_{i=1}^M \phat_i(x) \, dx \bigg)^{-1} < \infty
\end{equation}
may in general have an infinite expectation. This may happen because the estimators in the above product may decay too quickly in $x$ variable and the sets of $x$ with the most `mass' for each $p_i$ may have no common  intersection. This potentially may lead to small values of $\lhat$  and hence large $\chat$. To avoid this situation one would need to chose the kernel $K$ in appropriate way and establish the finiteness of the expectation of $\chat$.

\par

In this article we do not investigate this. Instead, we will show that one can replace $\MISE$ by an equivalent functional which is well-defined and finite on the whole sample space $\Omega$ and that there exists a sequence of smaller sample subspaces $\Omega_n$ with $\PP(\Omega_n) \to 1$	, on which the new functional is asymptotically {\it equivalent} to $\MISE[p,\phat]$ restricted to $\Omega_n$. We then analyze the new functional and investigate its optimal parameters.

\subsection{Preliminary estimates}

\par

\begin{lemma}[\textbf{covariance}]\label{covPStarLemma}
Let $\phat^*(x)$ be an estimator of the form \eqref{estmodel} where the vector of sample sizes $\NN(n)$ and bandwidth vector $\hh(n)$ satisfy \eqref{unifconvNh}. Then
\begin{equation}
	\Cov[\phat^*(x),\phat^*(y)]=\EXP[\phat^*(x)\phat^*(y)]-\EXP[\phat^*(x)]\EXP[\phat^*(y)]
\end{equation}
satisfies the estimates
\begin{equation}\label{covPstarEst}
\begin{aligned}
	|\Cov[\phat^*(x),\phat^*(y)]|&\leq \frac{C_{abs}}{\|\NN\|\|\hh\|}\\
	\left|\iint \Cov[\phat^*(x),\phat^*(y)]\,dxdy\right|&\leq \frac{C_{int}}{\|\NN\|}
\end{aligned}
\end{equation}
for some constants $C_{abs},C_{int}>0$ independent of $n$.

\begin{proof}
We can expand the product as follows
\[
\begin{aligned}
	&\prod_{i=1}^M \EXP[\phat_i(x)\phat_i(y)]-\prod_{i=1}^M \EXP[\phat_i(x)]\EXP[\phat_i(y)]\\
	&\qquad= \sum_{j=1}^M \left(\rule{0pt}{1em}\EXP[\phat_j(x)\phat_j(y)]-\EXP[\phat_j(x)]\EXP[\phat_j(y)]\right)\left(\prod_{i=1}^{j-1} \EXP[\phat_i(x)\phat_i(y)] \right) \left(\prod_{i=j+1}^{M} \EXP[\phat_i(x)]\EXP[\phat_i(y)]\right)
\end{aligned}
\]
where the products with the top index smaller than the bottom index should be taken as having the value one.

\par

We next observe that, according to \eqref{biasexp}, for each $i \in \{1,\dots,M\}$
\[
	|\EXP[\phat_i(x)]\EXP[\phat_i(y)]|\leq C\left(1+\frac{k_2 h_i^2}{2}+\frac{k_3 h_i^3}{6}\right)^2\,.
\]
Also Lemma \ref{expProdLemma} implies that
\[
\Big|\EXP[\phat_i(x)\phat_i(y)]\Big|\leq C\left(1+\frac{k_2 h_i^2}{2}+\frac{k_3 h_i^3}{6}\right)^2 +\frac{C}{N_i h_i} +  \frac{C}{N_i}\left(1+\left(1+\frac{k_2 h_i^2}{2}+\frac{k_3 h_i^3}{6}\right)^2\right)
\]

Then we conclude that for some $C_{\EXP} \geq 0$
\[
\begin{aligned}
	|\EXP[\phat_i(x) \phat_i(y)]|, |\EXP[\phat_i(x)]\EXP[\phat_i(y)]|&\leq C_{\EXP} < \infty\,, \quad \text{for all} \quad x,y \in \RR\,.
\end{aligned}
\]


Therefore, by Lemma \ref{expProdLemma} we obtain the estimate
\[
|\Cov[\phat^*(x),\phat^*(y)]|\leq M\,C\left(\frac{1}{\|\NN\|\|\hh\|} +\frac{1}{\|\NN\|}\left(1+\left(1+\frac{k_2 \|\hh\|^2}{2}+\frac{k_3 \|\hh\|^3}{6}\right)^2\right)\right)
\]
for some appropriate constant $C$, which gives \eqref{covPstarEst}$_1$.

The integral of $\Cov[\phat^*(x),\phat^*(y)]$ is also finite. Using the result of Lemma \ref{expProdLemma} and the hypothesis \eqref{pcond2}
\[
\begin{aligned}
&\iint \Big| \Cov\big[\phat^*(x),\phat^*(y)\big] \Big| \,dxdy\\
&\leq C_{\EXP}^{M-1}\sum_{i=1}^M \iint \left|\rule{0pt}{1em}\EXP[\phat_j(x)\phat_j(y)]-\EXP[\phat_j(x)]\EXP[\phat_j(y)]\right|\,dxdy\\
&\leq C_{\EXP}^{M-1}\sum_{i=1}^M \iint \left(\frac{1}{N_i}p_i(x)\frac{1}{h_i}K_2\left(\frac{x-y}{h_i}\right)+|E_{\Pi,i}(x,y)|\right)\,dxdy \\
&\leq C_{\EXP}^{M-1}\frac{M}{\|\NN\|}\left(2+C\left(k_1 + \frac{k_2 \|\hh\|^2}{2}+\frac{k_3 \|\hh\|^3}{6}\right)\right),\\
\end{aligned}
\]
Where at the last step we used the facts that $\frac{1}{h}K_2\left(\frac{x-y}{h}\right)$ is a probability density function in $y$ for any fixed $x$ and  $p_i(x)$ is also a probability density function.

\end{proof}
\end{lemma}


\begin{lemma}\label{lambdaVarLemma}Let $\phat^*(x)$ be an estimator of the form \eqref{estmodel} where the vector of sample sizes $\NN(n)$ and bandwidth vector $\hh(n)$ satisfy \eqref{unifconvNh}. Then following identity and the estimate holds
\[
\begin{aligned}
	\VAR[\lhat-\lambda]  = \VAR\Big[\int \phat^*(x)\, dx - \int p^*(x)\, dx \Big] \leq \frac{C_{int}}{\|\NN\|} <\infty,
\end{aligned}
\]
where $C_{int}> 0 $ is defined in  \eqref{covPstarEst}.
\begin{proof}
Since $\lambda$ is constant we have
\[
\begin{aligned}
	\VAR[\lhat-\lambda]&= \EXP\big[\lhat-\EXP[\lhat]\big]^2\\
		&=\EXP\left[\int_{\RR}\phat^*(x)-\EXP[\phat^*(x)] \, dx\right]^2\\
		&= \EXP\left[\int \left(\phat^*(x) -\EXP [\phat^*(x)]\right) dx  \cdot \int \left(\phat^*(y) -\EXP[\phat^*(y)]\right)dy\right]\\
		&= \iint \Big(\EXP \left[\phat^*(x)\phat^*(y)\right] -\EXP [\phat^*(x)]\EXP [\phat^*(y)]\Big)dx\,dy \leq \frac{C_{int}}{\|\NN\|},
\end{aligned}
\]
where the last inequality is from Lemma \ref{covPStarLemma}.
\end{proof}
\end{lemma}

\begin{lemma}\label{lambdaConvLemma}
	Let $\phat^*(x)$ be an estimator of the form \eqref{estmodel} where the vector of sample sizes $\NN(n)$ and bandwidth vector $\hh(n)$ satisfy
	\eqref{unifconvNh}. Then for any $\alpha \in (0,1]$
	 \begin{equation}\label{lcnvrgrate2}
	 	 \PP\bigg(
	 	   \bigg\{\omega:|\EXP{\lhat}-\lhat(\omega; \NN(n),\hh(n))| > \frac{{\lambda}}{\sqrt{2} {\|\NN\|}^{\frac{1-\alpha}{2}}}
	 	  	   \bigg\}
	 	 \bigg)
	 	     \leq \frac{2 C_{int}}{\lambda^2{\|\NN\|}^{\alpha}}\,.
	 \end{equation}
Moreover, for any $\alpha$ satisfying
\[
\max(0,1-4\alpha_0)<\alpha<1\,,
\]
where $\alpha_0$ is defined in \eqref{unifconvNh}\,, we have
\begin{equation}\label{Omegan1}
 \PP    \bigg\{
 					\Big|\frac{\lhat}{\lambda}-1\Big| > \frac{1}{\|\NN\|^{\frac{1-\alpha}{2}}}
		  	   \bigg\} \leq \frac{2 C_{int}}{\lambda^2 \|\NN\|^{\alpha}}\,.
\end{equation}
for all sufficiently large $n$.

\begin{proof} By Lemma \ref{lambdaVarLemma} and Chebyshev inequality we obtain
\[
\begin{aligned}
& \PP
		  \bigg\{\big|\lhat-\EXP[\lhat]\big|^2> \frac{\lambda^2}{2{\|\NN\|}^{1-\alpha}}
		  	   \bigg\}\\
		 & \qquad \leq \PP \bigg\{\big|\lhat-\EXP[\lhat]\big|^2> \VAR(\lhat) \frac{\lambda^2{\|\NN\|}^{\alpha}}{2 C_{int}}
		  	   \bigg\} \leq \frac{2C_{int}}{\lambda^2{\|\NN\|}^{\alpha}}\,.
\end{aligned}
\]

\par

Recall next that
\[
|\EXP(\lhat)-\lambda| = \Big| \int \Big(\EXP[\phat^*(x)]-p^*(x)\Big) \, dx \Big| \leq \int |\bias[\phat^*]| \, dx \leq C\|\hh\|^2
\]
where $C$ is independent of $\hh$. According to \eqref{unifconvNh} we have $\|\hh(n)\|\leq A \|\NN\|^{-\alpha_0}$ for some $\alpha_0 \in(0,1)$. Fix an arbitrary $\alpha$ that satisfies
\[
\max(0,1-4\alpha_0)<\alpha<1 \quad \text{so that} \quad 4\alpha_0>1-\alpha.
\]
Then
\begin{equation*}
{\|\hh\|^2}{\|\NN\|^{\frac{1-\alpha}{2}}} \leq
{A\|\NN\|^{-2\alpha_0}}{\|\NN\|^{\frac{1-\alpha}{2}}}
= A\|\NN\|^{\frac{-4\alpha_0+(1-\alpha) }{2}} \to 0 \quad \text{as $n \to \infty$}
\end{equation*}
Thus there exists $n_0$ such that
\begin{equation}\label{hbound}
C\|\hh(n)\|^2 < \frac{\lambda}{4}\|\NN(n)\|^{-\frac{(1-\alpha)}{2}} \quad \text{for all $n >n_0$}\,.
\end{equation}

By the triangle inequality we have
\[
\begin{aligned}
& \big|\lhat - \EXP \lhat \big| > \big|\lhat - \lambda  \big| - \big|\lambda  - \EXP \lhat  \big|
 > \big|\lhat - \lambda  \big| - \frac{\lambda}{4} \|\NN\|^{-\frac{(1-\alpha)}{2}}
\end{aligned}
\]
and hence for every
\begin{equation}\label{tempest1}
\omega_0 \in \Big\{\omega: \, |\lhat({\omega}) - \lambda| > \frac{\lambda}{\|\NN\|^{\frac{1-\alpha}{2}}}  \Big\}
\end{equation}
we have
\begin{equation}\label{tempest2}
\begin{aligned}
\big|\lhat(\omega_0) - \EXP \lhat \big| & > \big|\lhat(\omega_0) - \lambda  \big| - \frac{\lambda}{4} \|\NN\|^{-\frac{(1-\alpha)}{2}} > \frac{3\lambda}{4} \|\NN\|^{-\frac{(1-\alpha)}{2}}>\frac{\lambda}{\sqrt{2}} \|\NN\|^{-\frac{(1-\alpha)}{2}}\,.
\end{aligned}
\end{equation}
Then \eqref{tempest1} and \eqref{tempest2} we obtain
\begin{equation*}
\Big\{{\omega}: \, |\lhat({\omega}) - \lambda| > \frac{\lambda}{\|\NN\|^{\frac{1-\alpha}{2}}}  \Big\}	 \subset
\Big\{{\omega}: \, |\lhat({\omega}) - \EXP \lhat| > \frac{\lambda}{\sqrt{2}\|\NN\|^{\frac{1-\alpha}{2}}}  \Big\}	
\end{equation*}
and hence
\begin{equation}
\begin{aligned}
& \PP\Big\{{\omega}: \, |\lhat({\omega}) - \lambda| > \frac{\lambda}{\|\NN\|^{\frac{1-\alpha}{2}}}  \Big\} \\
&\qquad \leq \PP \Big\{{\omega}: \, |\lhat({\omega}) - \EXP \lhat | > \frac{\lambda}{\sqrt{2}\|\NN\|^{\frac{1-\alpha}{2}}}  \Big\}  \leq \frac{2 C_{int}}{\lambda^2 \|\NN\|^{\alpha}}\,.
\end{aligned}
\end{equation}
\end{proof}
\end{lemma}

\subsection{Functional equivalent to $\MISE$}

As it was pointed the functional $\MISE$ defined in \eqref{misefull} is not defined in the whole space $\Omega$ because the reciprocal of the renormalization random variable $(\lhat)^{-1}$ may in general have en infinite expectation.

The event space $\Omega_n$ insures that the constant $\chat$ has a finite expectation and stays close to the true normalization constant $c$. However, even on this smaller and safer space the functional $\MISE[\phat,p]$ is rather difficult to analyze. To help resolve this issue we introduce a functional that is asymptotically equivalent to $\MISE$ on the space $\Omega_n$
\begin{definition}
\begin{equation}\label{misebar}
	\overline{\MISE}=\EXP\left[\left(\frac{\lhat}{\lambda}\right)^2\int_{\RR}(\phat(x)-p(x))^2dx\,\right]
\end{equation}
\end{definition}
The equivalence follows from the definition of the space $\Omega_n$
\begin{proposition}
The functional ${\overline{\MISE}}$ is asymptotically equivalent to $\MISE$ on smaller events $\Omega_n$ uniformly in $n$, that is
\begin{equation}\label{proplimit}
\lim_{\|\NN(n)\| \to \infty} \frac{\overline{\MISE}[p(x),\phat(x;\omega)|\omega\in\Omega_n\big]}{\MISE[p(x),\phat(x;\omega)|\omega\in\Omega_n\big]}=1
\end{equation}
where 
\begin{equation}\label{Omegan}
\Omega_{n} = \bigg\{\omega \in \Omega:\,
 					|\lhat-\lambda| \leq \frac{\lambda}{\|\NN\|^{\frac{1-\alpha}{2}}}
		  	   \bigg\} \quad \text{with} \quad \PP(\Omega_n) \geq 1 - \frac{C}{\|\NN\|^{\alpha}}
\end{equation}
and $\alpha$ is a fixed constant satisfying $1>\alpha>\min(1-4\alpha_0,0)$.


\begin{proof}

	Observe that
	\[
	\begin{aligned}
		&\overline{\MISE}[p(x),\phat(x)|\Omega_n\big]=\frac{1}{\PP(\Omega_n)}\int_{\Omega_n}\left(\frac{\lhat}{\lambda}\right)^2\int_{\RR}(\phat(x, \omega)-p(x))^2dx\,\PP(d\omega)\\
		&\qquad = \frac{1}{\PP(\Omega_n)}\bigg(\int_{\Omega_n}\left[\left(\frac{\lhat}{\lambda}-1\right)^2+2\left(\frac{\lhat}{\lambda}-1\right)+1\right]\int_{\RR}(\phat(x, \omega)-p(x))^2dx\,\bigg)\PP(d\omega).
	\end{aligned}	
	\]
Then by \eqref{Omegan} we obtain that
	\[
		\overline{\MISE}[p(x),\phat(x)|\Omega_n\big]=(1+\eps(n))\MISE[p(x),\phat(x)|\Omega_n\big]
	\]
	where
	\[
	|\eps(n)|\leq\frac{C}{\|\NN\|^{\frac{1-\alpha}{2}}},
	\]
	for some constant $C>0$ independent of $n$. This implies \eqref{proplimit}.
\end{proof}
\end{proposition}
One of the positive side effects we must mention is that the functional defined in \eqref{misebar} is not only easier to analyze but also it is defined throughout the whole space $\Omega$. We take advantage of this fact and continue the discussion with expectations taken over the whole unrestricted space.

With the slight modification of the functional we can now extract the leading order part
\begin{theorem}\label{AMISEplmm}
The distance functional $\overline{\MISE}$ can be represented as
\begin{equation}\label{misebdecomp}
\begin{aligned}
	&\overline{\MISE}[p,\phat,\NN,\hh]=\overline{\AMISE}[p,\phat,\NN,\hh] + \mathcal{E}(\NN,\hh)
\end{aligned}
\end{equation}
where the leading term 
\begin{equation}\label{amiseb}
\begin{aligned}
	\overline{\AMISE}[p,\phat,\NN,\hh]&:=\left(\int B(x) \,dx\right)^2\int\big(c p^*(x)\big)^2\,dx\\
	&\qquad+ \int\left(B(x) \right)^2dx+\int_{\RR}\left(V(x)\right)dx\int_{\RR}K^2(t)\,dt \\
	&\qquad- 2\iint B(y)\, B(x)\,cp^*(x)\, dx\,dy\\
	B(x)&=\frac{ck_2}{2}  \sum_{m=1}^{M}\left[h_m^2 p_m''(x)\prod_{\substack{k=1\\k\neq m}}^M p_k(x)\right]\\
	V(x)&=\sum_{m=1}^{M}\left[\frac{p_m}{N_m h_m}\prod_{\substack{k=1\\ k \neq m }}^M p_k^2(x)\right]\\
\end{aligned}
\end{equation}
and the error term $\mathcal{E}$ satisfies
\begin{equation}
\mathcal{E}(\NN,\hh) = o\Big( ||\hh||^4 + \frac{1}{\|\NN\|\|\hh\|}\Big)
\end{equation}
as $\hh \to 0$, $\NN \to \infty$, and $(\|\NN\|\|\hh\|)^{-1} \to 0$.

\begin{proof}
We can divide the functional $\overline{\MISE}$ into three components
\begin{equation}\label{decomp}
\begin{aligned}
	\overline{\MISE}[p,\phat]&=J_1 + J_2 + J_3\\
		&=c^2\EXP[(\lambda-\lhat)^2]\int_{\RR}\left(p(x)\right)^2dx\\
		&\qquad + c^2\EXP\int_{\RR}(\phat^* - p^*)^2dx\\
		&\qquad - 2 c^2\EXP \int_{\RR}(\lhat-\lambda)(\phat^*-p^*) p(x)\,dx.
\end{aligned}	
\end{equation}

\par\smallskip

Our first step will be to express each term $J_i$, i=1,\dots,3, as a sum of a higher order term and the term containing a bias, variance or their combination. We then will use the results of the previous section and the appendix to obtain a leading part of each term.

First observe that
\[
\begin{aligned}
\EXP[(\lambda-\lhat)^2]&=\EXP[(\lambda-\EXP[\lhat])^2]+\EXP[(\EXP[\lhat]-\lhat)^2]\\
& = \bigg(\EXP\bigg[\int p(x)- \phat(x) \, dx \bigg]\bigg) ^2 + \EXP[(\EXP[\lhat]-\lhat)^2].
\end{aligned}
\]
The second term turns out to be of higher order. This can be seen from the following estimate
\begin{equation}
\begin{aligned}
	&\EXP[(\lhat-\EXP[\lhat])^2]=\EXP\left[\int \left(\phat^* -\EXP \phat^*\right)dx \right]^2\\
	&\qquad = \EXP\left[\int \left(\phat^* -\EXP \phat^*\right)dx \cdot \int \left(\phat^* -\EXP \phat^*\right)dx\right]\\
	&\qquad = \EXP\left[\int \left(\phat^*(x) -\EXP \phat^*(x)\right) dx \cdot \int \left( \phat^*(y) -\EXP \phat^*(y)\right)dy\right]\\
	&\qquad = \iint \left(\EXP\left[\phat^*(x)\phat^*(y)\right] -\EXP [\phat^*(x)]\EXP [\phat^*(y)]\right)dx\,dy\leq\frac{C_1}{\|\NN\|}
\end{aligned}
\end{equation}
where the last inequality follows from Lemma \ref{covPStarLemma}.

Thus, we conclude
\[
	J_1=c^2 \left(\int \bias[p^*,\phat^*]\,dx\right)^2\int\big(p(x)\big)^2\,dx +E_1\quad\text{where } |E_1|\leq \frac{C}{\|\NN\|}.
\]

From \eqref{multimiseprod} we have that
\[
	J_2 =  c^2 \int \bigg(\bias^2[p^*,\phat^*]+\VAR[\phat^*] \bigg)\,dx.
\]

The term $J_3$ can be expressed as
\[
\begin{aligned}
	&J_3 = c^2\EXP_n\iint \big(\phat^*(y)-p^*(y)\big)\Big(\big(\phat^*(x)-p^*(x)\big) p(x) \Big)\, dy dx\\
	&\qquad=  c^2\iint \bias[\phat^*(y)]\bias[\phat^*(x)]\,p(x)\, dx\,dy\\
	&\qquad\qquad + c^2\EXP \iint  \big(\EXP[\phat^*(y)]-\phat^*(y)\big)\big(\EXP[\phat^*(x)]-\phat^*(x)\big) p(x) \, dy dx
\end{aligned}
\]
Since $p^*(x)$ is uniformly bounded, Lemma \ref{covPStarLemma} implies that the last term in the above identity satisfies
\[
	\left|c^2\EXP \iint  \big(\EXP[\phat^*(y)]-\phat^*(y)\big)\big(\EXP[\phat^*(x)]-\phat^*(x)\big) c p^*(x) \, dy dx\right|\leq \frac{C}{\|\NN\|}.
\]
This gives
\[
\begin{aligned}
	&J_3 =  c^2\iint \bias[\phat^*(y)]\bias[\phat^*(x)]\,p(x)\, dx\,dy + E_3, \quad |E_3|<\frac{1}{\|\NN\|}\,.
\end{aligned}
\]

Combining the above estimates gives
\begin{equation}\label{tempdecom}
\begin{aligned}
\overline{\MISE}\big[p,\phat\big] &=  c^2 \left(\int \bias[p^*,\phat^*]\,dx\right)^2\int\big(p(x)\big)^2\,dx\\
	&\qquad + c^2 \int\bias^2[p^*,\phat^*]+\VAR[\phat^*]\,dx\\
	&\qquad -2 c^2\iint \bias[\phat^*(y)]\bias[\phat^*(x)]\,p(x)\, dx\,dy + E_M, \quad |E_M|\leq\frac{C}{\|\NN\|}
\end{aligned}
\end{equation}

Applying the results of Lemma \ref{miltibiasformula} and Lemma \ref{varmulti} to the  identity \eqref{tempdecom} leads to \eqref{misebdecomp} and \eqref{amiseb} and this finishes the proof.

\end{proof}
\end{theorem}

\subsection{Numerical optimization scheme for optimal bandwidth}\label{numericalscheme}
In the absence of knowledge of probability density functions $p(x)$ and $p_m(x)$, it may seem that the formula \eqref{amiseb} has little practical use. However, we can replace the densities with their approximations $\phat(x)$ and $\phat_m(x)$. This will turn the quantity $\AMISEB$ into a function of the form
\begin{equation}\label{amisebhat}
{\AMISEB}(\hh)=\sum_{i,j=1}^M h_i^2h_j^2 \beta_{i,j} + \sum_{i=1}^M\frac{\nu_i}{h_i}
\end{equation}
where
\begin{equation}\label{amisebcoeffs}
\begin{aligned}
\beta_{i,j}&=\frac{\chat^2k^2_2}{4}  \int_\RR \phat_i''(x)\prod_{\substack{k=1\\k\neq i}}^M \phat_k(x)\,dx \int_\RR \phat_j''(x)\prod_{\substack{k=1\\k\neq j}}^M \phat_k(x)\,dx\int_\RR \phat^2(x)\,dx\\
&\qquad + \frac{\chat^2k^2_2}{4}\int_\RR \left(\phat_i''(x)\prod_{\substack{k=1\\k\neq i}}^M \phat_k(x)\right)\left( \phat_j''(x)\prod_{\substack{k=1\\k\neq j}}^M \phat_k(x)\right)dx\\
&\qquad - \frac{\chat^2k^2_2}{2}\int_\RR \phat_i''(y)\prod_{\substack{k=1\\k\neq i}}^M \phat_k(y)\,dy \int_\RR \phat_j''(x)\prod_{\substack{k=1\\k\neq j}}^M \phat_k(x) \phat(x)\,dx\\
\nu_i&=\int_\RR\frac{\phat_i}{N_i}\prod_{\substack{k=1\\ k \neq i }}^M \phat_k^2(x)\,dx\int_{\RR}K^2(t)\,dt
\end{aligned}
\end{equation}
This can be used to create an iterative algorithm that will yield a near optimal value for $\hh$. A possible implementation of such an algorithm is laid out in Algorithm \ref{minhalg}.

\begin{algorithm}
\caption{Locate optimal bandwidth vector $\hh=(h_1,\ldots,h_M)$}\label{minhalg}
\mbox{\textbf{input:} Samples $X_{i,m}$}\\
\mbox{\textbf{result:} Vector $\hh=(h_1,h_2,\ldots,h_M)$}
\begin{algorithmic}[1]
\State Initialize $k=0$, $\hh_0$ with $h_i=\hopt$ from \eqref{hoptbandnorm}
\Repeat
\State Increment $k$
\State Compute estimators $\phat_m$, $\phat^*$ and $\chat$
\State Compute $\beta_{i,j}$ and $\nu_i$ from \eqref{amisebcoeffs}
\State  Compute $\hh_{k+1}$ by executing a few steps of the gradient (or conjugate gradient) descent minimization algorithm applied to   $\AMISEB(\hh)$
\Until {$\|\hh_k-\hh_{k+1}\|$ is sufficiently small}
\end{algorithmic}
\end{algorithm}


The conditions, under which the iterative algorithm \ref{minhalg} will converge to the minimizing vector $\HOPT$, need to be thoroughly investigated. Such investigation is outside the limits of this publication and is one of the directions of the future work the authors consider.

\section{Examples}\label{examplesSec}
In a general setting, finding a bandwidth vector $\hh$ that minimizes \eqref{amiseb} would require solving a system of nonlinear equations, which would probably not have a closed form solution and require application of numerical methods.
In this section we discuss two special cases, for which closed form solutions can be obtained with relative ease.\

\subsection{$\overline{\AMISE}$ optimization for a symmetric case}
In this case we assume that all posterior densities for each subset of samples are the same, and that all subsets contain the same number of samples. In other words we employ the following assumptions
\begin{itemize}
	\item $p_1(x)=p_2(x)=\dots=p_M(x) = f(x)$.
	\item $N_1=N_2=\dots=N_M$, that is, $\NN=(n,n,\dots,n)$, for some $n \in \mathbb{N}$
\end{itemize}
In view of the symmetry, all components of the optimal bandwidth vector should be the same, that is $\hh=(h,h,\dots,h)$. Under these assumptions, the expression for $\overline{\AMISE}$ simplifies into
\begin{equation}\label{amisesymm}
\begin{aligned}
	&\overline{\AMISE}[p(x),\phat(x)|\NN,\hh]:=M^2\frac{c^2k^2_2 h^4}{4}  \left(\int p_1''(x)p_1^{M-1}(x)\,dx\right)^2\int\big(c p_1^M(x)\big)^2\,dx\\
	&\qquad + M^2\frac{c^2k^2_2h^4}{4} \int\left(p_1''(x)p_1^{M-1}(x)\right)^2dx+M\int_{\RR}\left(\frac{p_1^{2M-1}}{n h}\right)dx\int_{\RR}K^2(t)\,dt\\
	&\qquad - M^2\frac{c^3k^2_2h^4}{2} \iint \left(p_1''(y) p_1^{M-1}(y)\right)\left(p_1''(x) p_1^{2M-1}(x)\right)\, dx\,dy\\
\end{aligned}
\end{equation}
This expression achieves its minimum when $h=\HOPT$ where
\begin{equation}\label{hoptsym}
\HOPT={(4 n)^{-1/5}} \bigg(\frac{B(M)}{A(M)} \bigg)^{1/5} \\
\end{equation}
and the constants $A$ and $B$ are given by
\begin{equation}\label{ABdef}
\begin{aligned}
A(M)&=M\frac{c^2k^2_2}{4}  \Bigg[\left(\int_\RR p_1''(x)p_1^{M-1}(x)\,dx\right)^2\int_\RR\big(c p_1^M(x)\big)^2\,dx\\
&\qquad+ \int_\RR\left(p_1''(x)p_1^{M-1}(x)\right)^2dx-2c\iint_{\RR^2} \left(p_1''(y) p_1^{M-1}(y)\right)\left(p_1''(x) p_1^{2M-1}(x)\right)\, dx\,dy\Bigg]\\
B(M)&=c^2\int_{\RR}\left({p_1^{2M-1}}\right)dx\int_{\RR}K^2(t)\,dt.\\
\end{aligned}
\end{equation}
Forming the bandwidth vector $\HOPTAM=(\HOPT,\HOPT,\dots,\HOPT)$, should yield a smaller value for $\overline{\AMISE}$ than the one achieved with the conventional choice given in \eqref{introhoptimal1d}.

\subsection{$\overline{\AMISE}$ optimization for normal subset posterior densities}


Let us assume that all subsets of samples of $x$ satisfy
\begin{itemize}
	\item $p_m=\mathcal{N}(x,\mu,\sigma)$ is a normal distribution with the same mean and standard deviation for each $m=1,\ldots,M$
	\item $N_1=N_2=\dots=N_M$, that is, $\NN=(n,n,\dots,n)$, for some $n \in \mathbb{N}$.
\end{itemize}
Again, using symmetry argument, we look for the minimizer on the set of positive vectors $\hh=(h,h,\dots,h)$. In that case, the optimal $h=\hopt$ is computed by \eqref{hoptsym} where constants $A$ and $B$ are computed  by \eqref{ABdef} with $p_1(x)$ replaced by $\mathcal{N}(x,\mu,\sigma)$. This gives
\[
A(M)=\frac{3 }{32\pi^{1/2}M^{1/2}\sigma^5 }
\]
and
\[
B(M)=\frac{M}{ 2\pi^{1/2}\sqrt{2 M-1}}
\]
and hence the minimizer of the leading part is given by 
\begin{equation}\label{hoptbandnorm}
 {\hh}^{\rm opt}=(1,1,\dots,1)h^{\rm opt} \quad \text{with} \quad
 	\HOPT= \bigg( \frac{16}{9} \frac{M^{3}}{(2 M-1)} \bigg)^{1/10}\sigma n^{-1/5}\,.
 \end{equation}

Recall that $n$ is the number of samples that each subset contains and hence  the total number of samples for all subsets is given by $\|\NN\|_1=n\cdot M$. Thus, letting $M \to \infty$ we obtain
\begin{equation}\label{asybandnorm}
	\HOPT=
	\left((8/9)^{1/10}+O(M^{-1})\right) (n\, M)^{-1/5}\sigma\qquad \text{as\quad}M\to\infty.
\end{equation}

Setting $M=1$ in \eqref{hoptbandnorm} we once again obtain the bandwidth vector 
\begin{equation}\label{hepanechnikov}
\hh_0^{\rm opt}=(1,1,\dots) h^{\rm opt}_{M=1} \quad \text{with}	\quad \hopt_{M=1}=\left(\frac{4}{3}\right)^{1/5}\sigma n^{-1/5}
\end{equation}
where each component $\hopt_{M=1}$ is the optimal bandwidth parameter for the individual subset posterior density estimator. Thus the `intuitive' choice of the bandwidth vector as $\hoptv_0$ leads to a suboptimal approximation of $\phat(x)$.

\subsection{$\overline{\AMISE}$ optimization for gamma distributed subset posterior densities}


Let us assume that all subsets of samples of $x$ satisfy
\begin{itemize}
	\item $p_m=\Gamma(x,\alpha,\beta)$ is a gamma distribution where $\alpha$ and $\beta$ are the same for each $m=1,\dots,M$. 
	\item $N_1=N_2=\dots=N_M$, that is, $\NN=(n,n,\dots,n)$, for some $n \in \mathbb{N}$.
\end{itemize}
By symmetry argument, we look for the minimizer on the set of positive vectors $\hh=(h,h,\dots,h)$. By substituting $p_1(x)$ by $\Gamma(x,\alpha,\beta)$ in \eqref{hoptsym} and \eqref{ABdef} we can obtain formulas similar to the ones derived in the previous section. Evaluating the integrals is not very challenging, however the integration results in very bulky expressions. 
\begin{equation}\label{hoptgamma}
\begin{aligned}
h(n,M,\alpha)&=\frac{1}{(4n^2\pi)^{1/10}}\left(\frac{A}{B+C+D}\right)^{1/5}\\
A&=2^{2 ({\alpha}-1) M} (2 M-1)^{-2 {\alpha} M+{\alpha}+2 M-2} \Gamma ({\alpha}) \left(\frac{M}{\theta }\right)^{3 ({\alpha}-1) M-1} \\
&\qquad\times\theta ^{3 {\alpha} M-2 M+4} \Gamma (2 M {\alpha}-{\alpha}-2M+2)\\[3pt]
B&=\frac{({\alpha}-1)^2 (M-1)^2 M^2 \left(\frac{M}{\theta }\right)^{({\alpha}-1) M} \theta ^{{\alpha} M} \Gamma (2 ({\alpha}-1) M) \Gamma (({\alpha}-1) M-1)^2}{\Gamma (({\alpha}-1)M+1)^2}\\[3pt]
C&=2 (M ({\alpha} (4 (M-1) M+3)-4 (M-1) M-15)+9) \left(\frac{M}{\theta }\right)^{({\alpha}-1) M} \\
&\qquad\times\theta ^{{\alpha} M} \Gamma (2 ({\alpha}-1) M-3)\\[3pt]
D&=\frac{2 ({\alpha}-1) (M-1) (2 M-1) M^{({\alpha}-1) M+1} \theta ^M \Gamma (({\alpha}-1) M-1) \Gamma (2 ({\alpha}-1) M-1)}{\Gamma (({\alpha}-1) M+1)}
\end{aligned}
\end{equation}
It must be noted that this result is very different from the normal distribution one ,and the suggested values of $h$ are approximately thirty percent smaller than those in case of normal distribution even if the standard deviation of the samples are the same. This further necessitates the need for an easy-to-apply method for numerical approximation of the bandwidth vector $\hh$, as the KDE method even for very similar families of distributions (such as normal and gamma ones) achieves best performance for very different bandwidth values. We discussed one such possible numerical scheme in section \ref{numericalscheme}.

%
\subsection{Numerical experiments with normal subset posterior densities}

\subsubsection{Description of the experiment}
The numerical experiment is designed to investigate the location of the optimal bandwidth parameter by approximating the true value of $\MISE[p,\phat]$ by repeated simulation.  One iteration of the experiment generates $M$ subsets of a predetermined number of samples with $p_m=\mathcal{N}(x,0,1)$, $m=1,\ldots,M$. Then the approximation $\phat(x)$ is computed several times with varied bandwidth parameters $h$ and integrated square error $\mathrm{ISE}[p(x),\phat(x),h]$ is then computed via numerical integration. The iteration is repeated a thousand times to obtain an approximation of $\MISE[p(x),\phat(x),h]$ and its standard deviation. This process is repeated for varying sample sizes and numbers of subsets. 

Once the data is collected, the minimum of $\MISE[p(x),\phat(x),h]$ is located and the bandwidth parameter $h$ for which the minimum is obtained is recorded. Since $h$ computed this way is a random variable, the whole experiment is repeated a hundred times to compute the approximation of the expected value of $h$ that minimizes $\MISE[p(x),\phat(x),h]$ and its variance. 
\subsubsection{Numerical results}
The experiments we ran allow us to compare the behavior of $\MISE[p(x),\phat(x),\hh]$ when we select $\hh=\hoptv_0$ from \eqref{hepanechnikov} and when we select $\hh=\hoptv$ from \eqref{hoptbandnorm}. Figures \ref{figerrplot4} and \ref{figerrplot8} demonstrate that the latter choice is clearly a superior one.
%
%
%
 \begin{figure}[htb]
 \centering
	\begin{subfigure}[t]{0.4\textwidth}
		\centering
		\includegraphics[width=\textwidth]{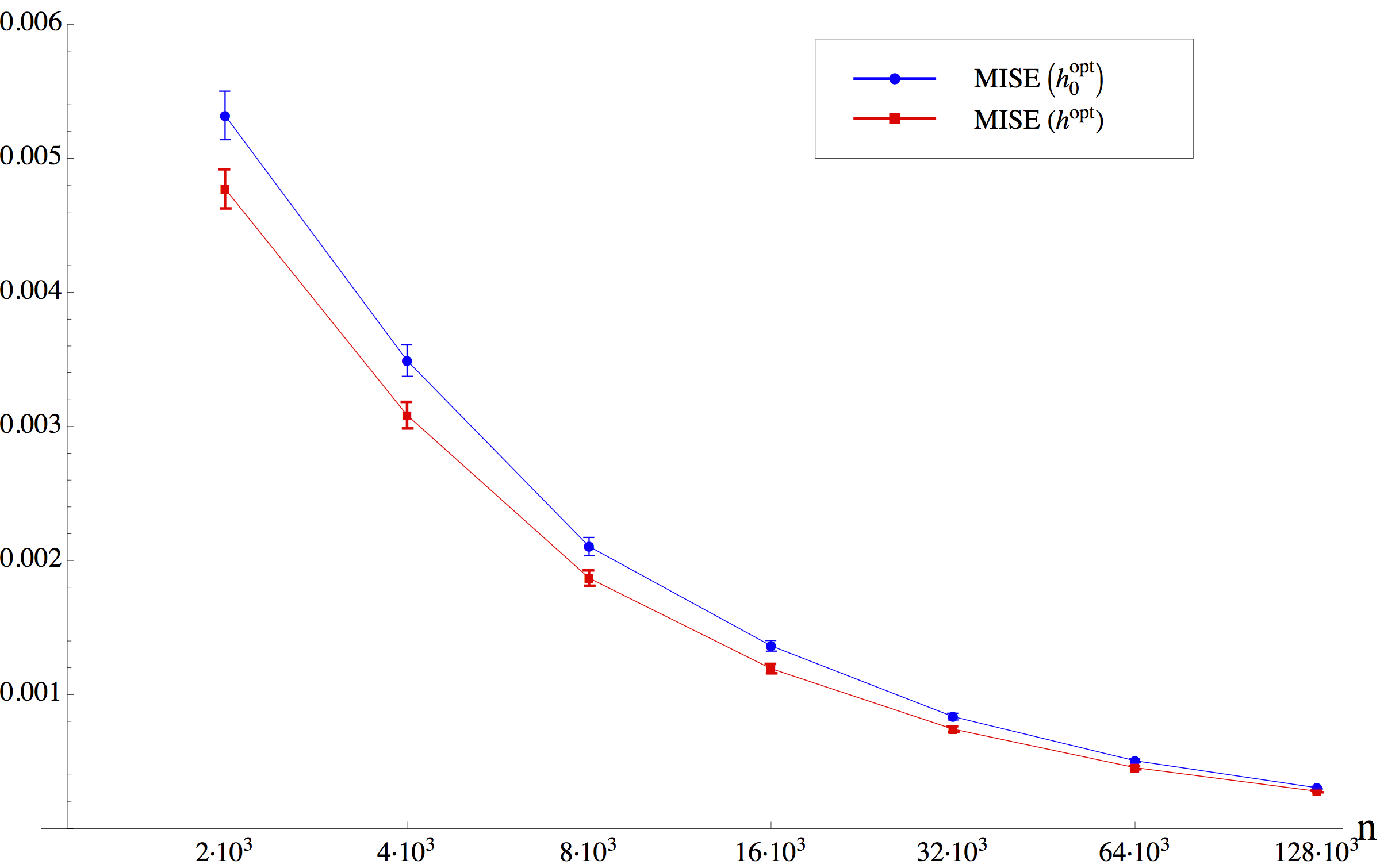}
		\caption{\scriptsize
		$M=4$}\label{figerrplot4}
	\end{subfigure}
	~~
	\begin{subfigure}[t]{0.4\textwidth}
		\centering
		\includegraphics[width=\textwidth]{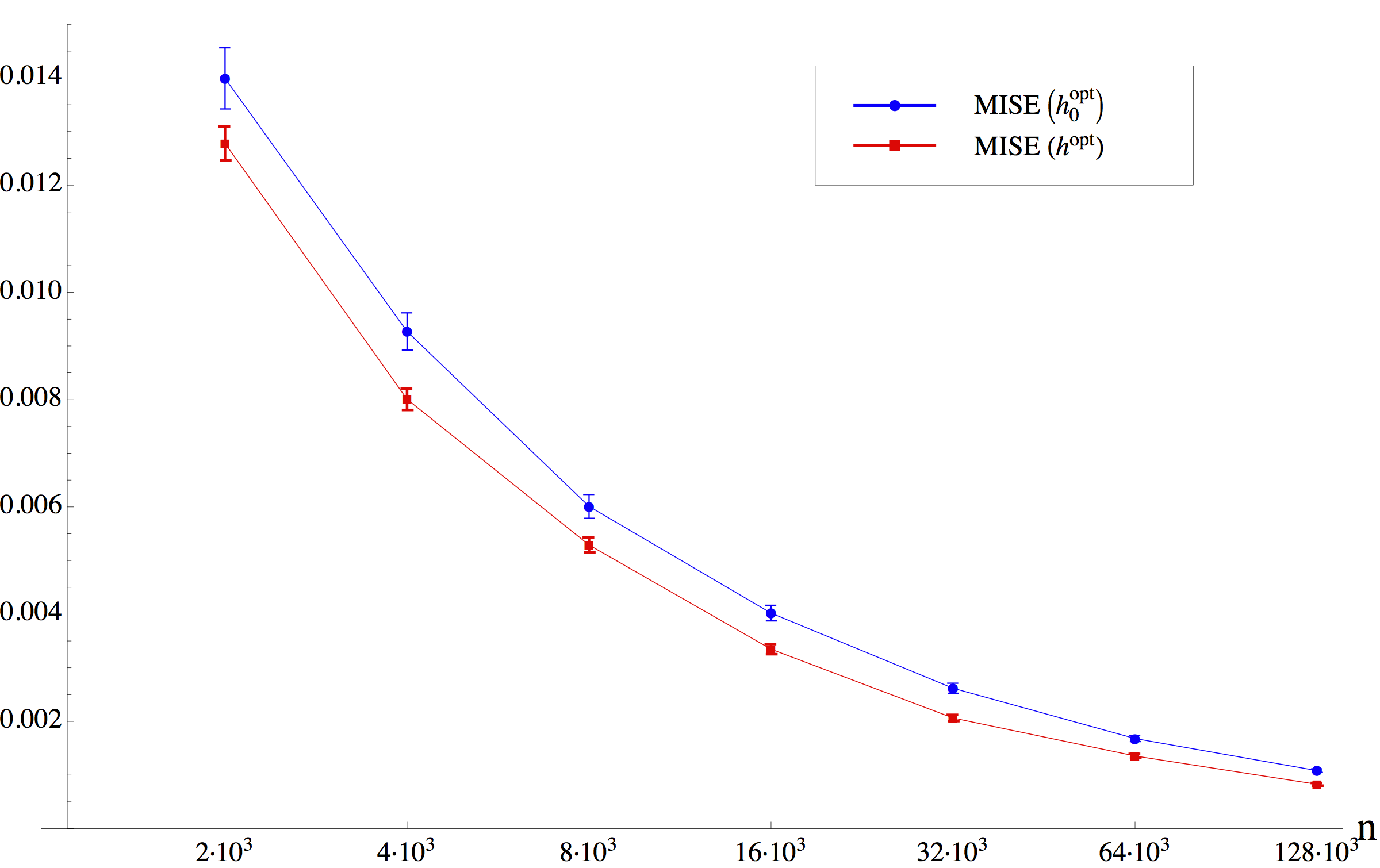}
		\caption{\scriptsize
		$M=8$}\label{figerrplot8}
	\end{subfigure}
	\caption{\footnotesize$\MISE[p,\phat,\NN(n),\hh]$ for $\hoptv$ and $\hoptv_0$.} 
\end{figure}
The rate of decay of the error is very close to $O(\|\NN\|^{-4/5})$, which is consistent with our calculations.\\

It must be noted, that the graphs are plotted at the theoretically optimal values of $\hh$, and the question of whether or not the error can be improved, must be addressed. Our experiment computes the values of $\MISE$ for a variety of values of $\hh$ and the bandwidth that produces the smallest error is indeed slightly different from our theoretical predictions. However, the discrepancy between them is negligible and it does become smaller as sample sizes increase. 

Let us define
 \[
\begin{aligned}
\hoptv_{\MISE} &= \mathrm{argmin}_{\hh\in\RR^M_{+}}\MISE[p^*, \phat^*; \NN,\hh]\\
	&=\mathrm{argmin}_{h\in\RR_{+}}\MISE[p^*, \phat^*; \NN,(h,h,\ldots,h)],\\
	&=h^{\mathrm{opt}}_{\MISE}\cdot(1,1,\ldots,1)
\end{aligned}
 \]
where the last two equalities hold in view of the symmetry assumption on $p^*$.

Figure \ref{figminplot} shows that the ratio of the numerically computed approximation of $\hoptv_{\MISE} $ to the theoretically  predicted value $\hoptv$ stays very close to one, which confirms validity of our approach.
\begin{figure}[htb]
 \centering
	\begin{subfigure}[t]{0.4\textwidth}
		\centering
		\includegraphics[width=\textwidth]{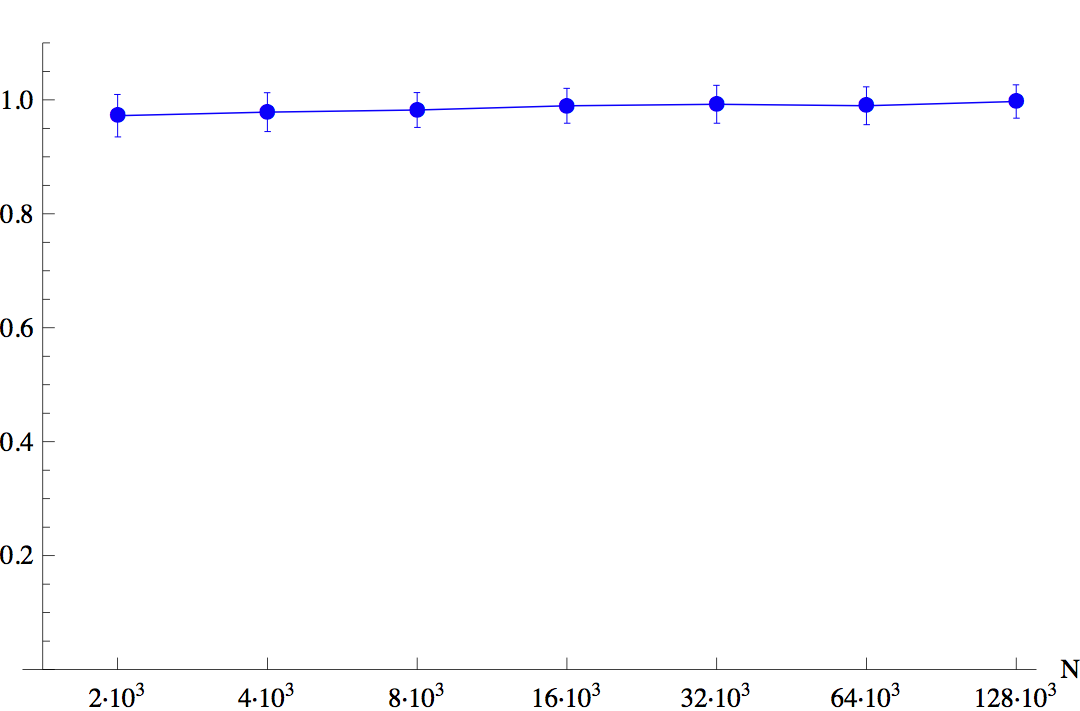}
		\caption{\scriptsize
		$M=4$}\label{figminplot4}
	\end{subfigure}
	~~
	\begin{subfigure}[t]{0.4\textwidth}
		\centering
		\includegraphics[width=\textwidth]{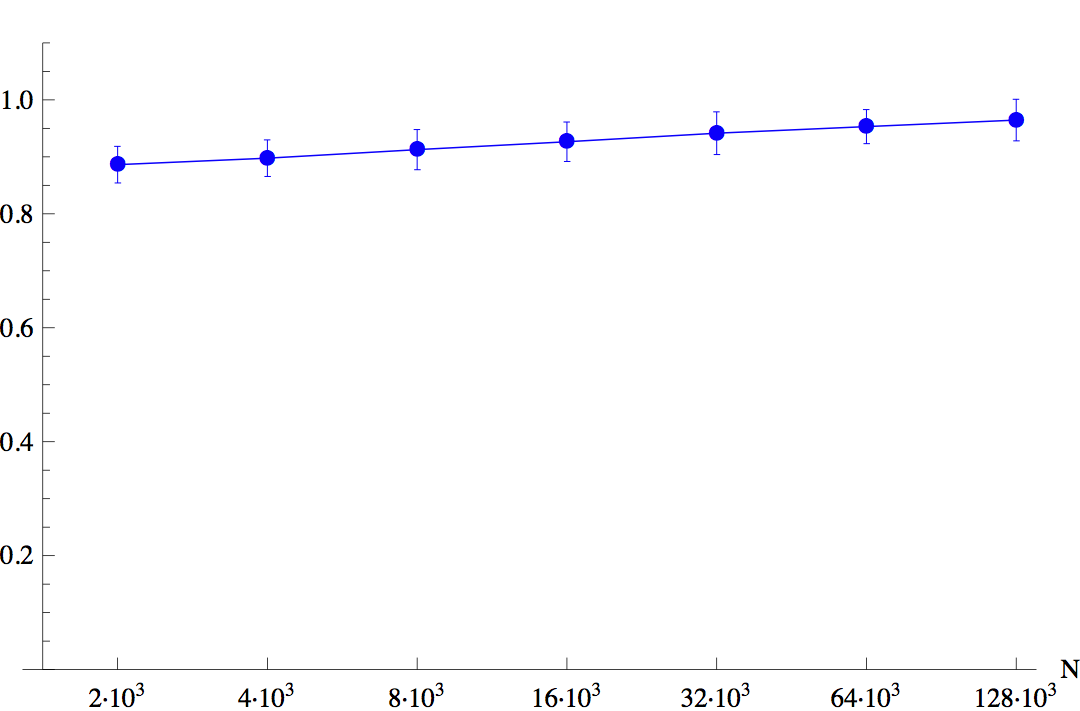}
		\caption{\scriptsize
		$M=8$}\label{figminplot8}
	\end{subfigure}
	\caption{\footnotesize The ratio $h^{\mathrm{opt}}/h^{\mathrm{opt}}_{\MISE}$
	for different subset configurations.}\label{figminplot}
\end{figure}

\subsection{Numerical experiments with gamma distributed subset posterior densities}

\subsubsection{Description of the experiment}
The numerical experiment mimics the one with normally distributed samples, with the only difference that this experiment generates samples distributed with $\Gamma(x,\alpha=3,\beta=3)$. 

\subsubsection{Numerical results}
The results of the experiments replicate the same behavior for gamma distributed samples. We must note that the location of the optimal bandwidth parameter is significantly different that in the case of normally distributed samples. Nevertheless, the results are clearly show the advantage of our choice of $\hh$, which is demonstrated in Figures \ref{figerrplotg4} and \ref{figerrplotg8}.
%
%
%
 \begin{figure}[htb]
 \centering
	\begin{subfigure}[t]{0.4\textwidth}
		\centering
		\includegraphics[width=\textwidth]{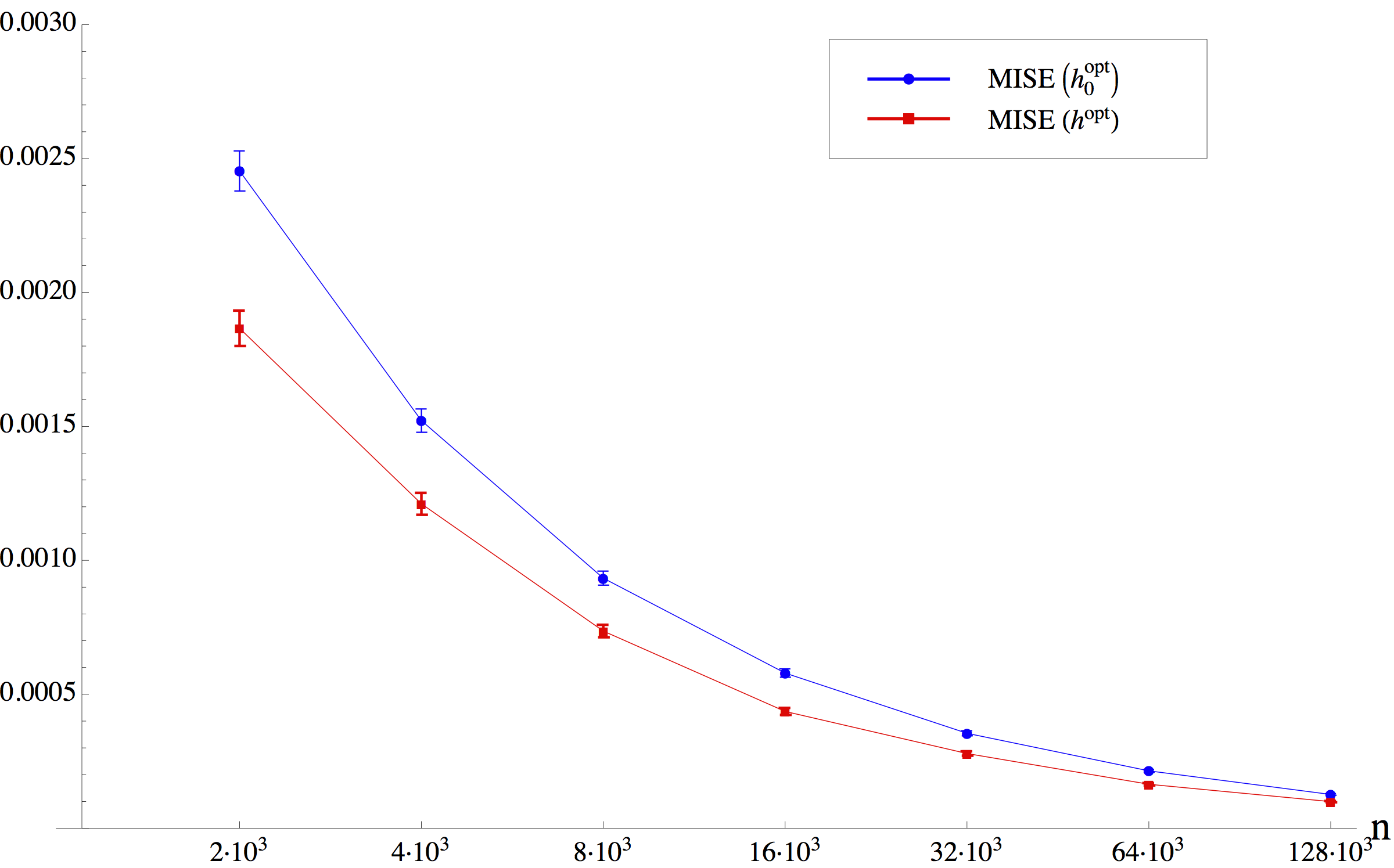}
		\caption{\scriptsize
		$M=4$}\label{figerrplotg4}
	\end{subfigure}
	~~
	\begin{subfigure}[t]{0.4\textwidth}
		\centering
		\includegraphics[width=\textwidth]{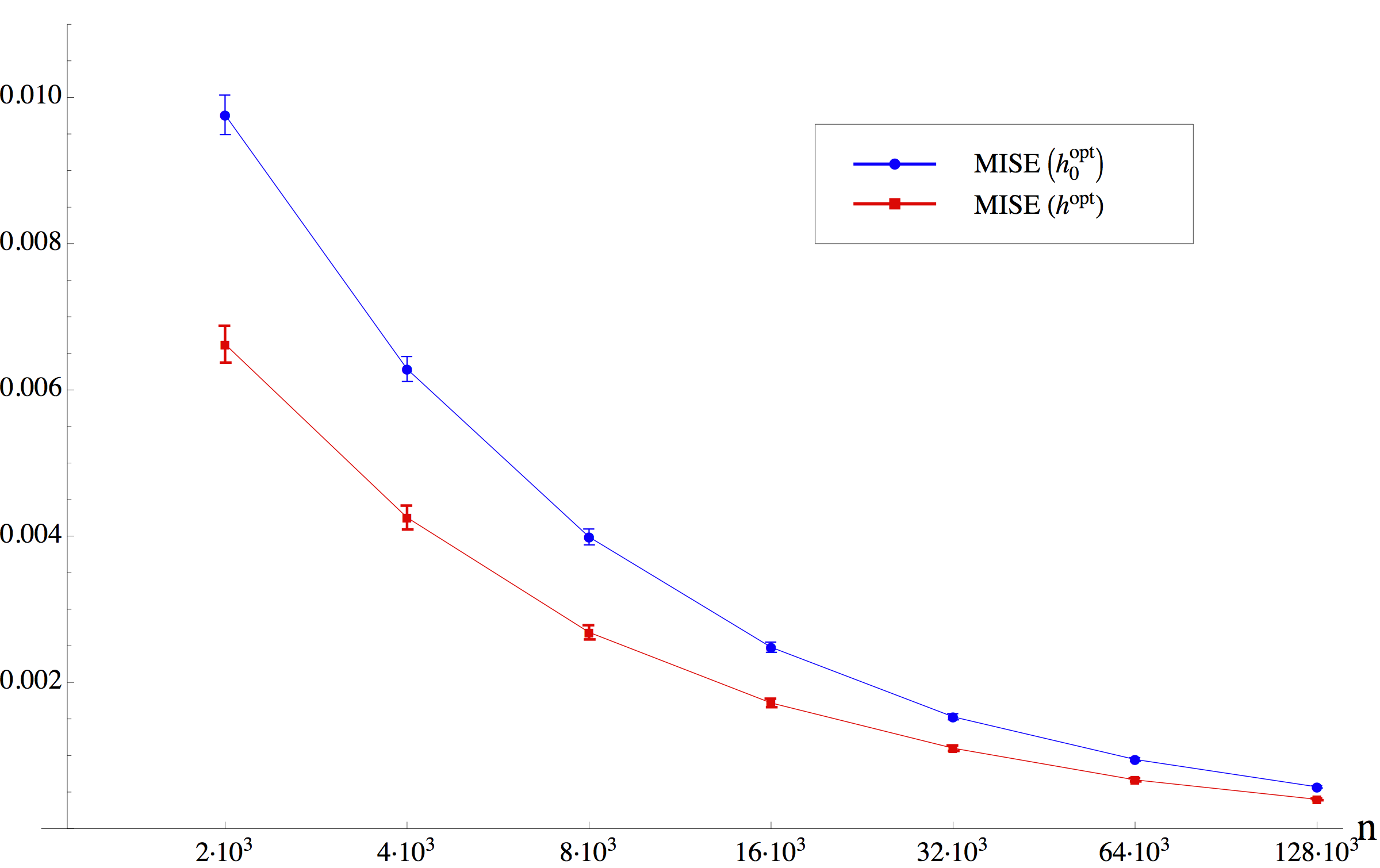}
		\caption{\scriptsize
		$M=8$}\label{figerrplotg8}
	\end{subfigure}
	\caption{\small$\MISE[p,\phat,\NN(n),\hh]$ for $\hoptv$ and $\hoptv_0$.}
\end{figure}

Just as before, our experiment verify that formula \eqref{hoptgamma} yields near optimum values of $\MISE$, see Figure \ref{figminplotg}.
\begin{figure}[htb]
 \centering
	\begin{subfigure}[t]{0.4\textwidth}
		\centering
		\includegraphics[width=\textwidth]{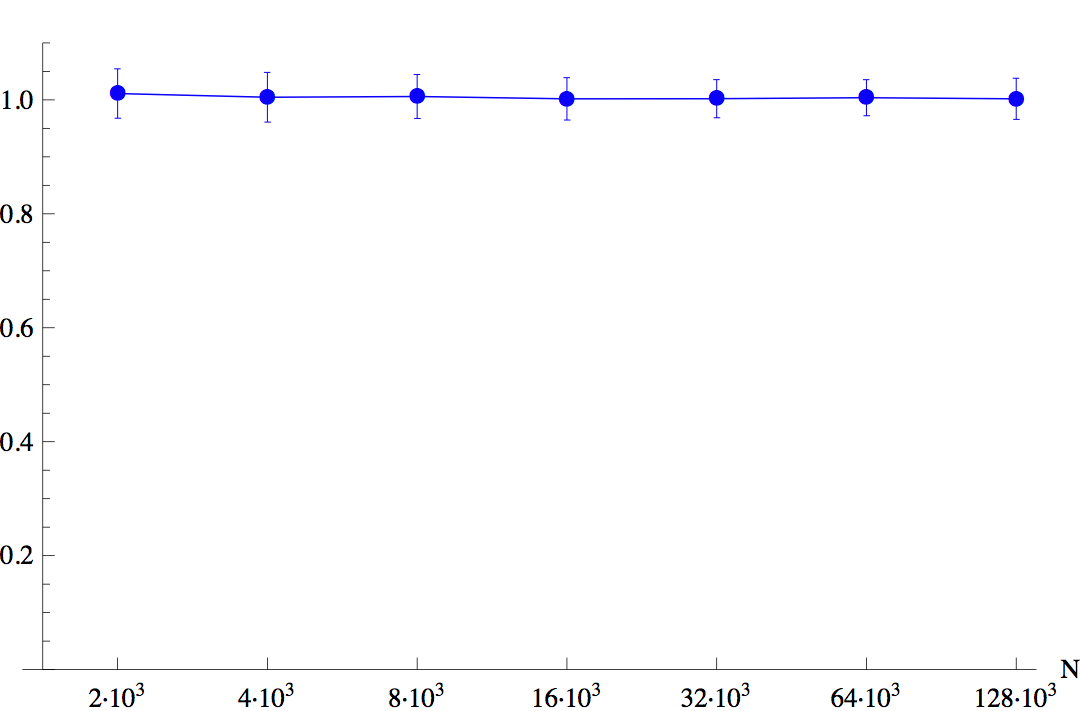}
		\caption{\scriptsize
		$M=4$}\label{figminplotg4}
	\end{subfigure}
	~~
	\begin{subfigure}[t]{0.4\textwidth}
		\centering
		\includegraphics[width=\textwidth]{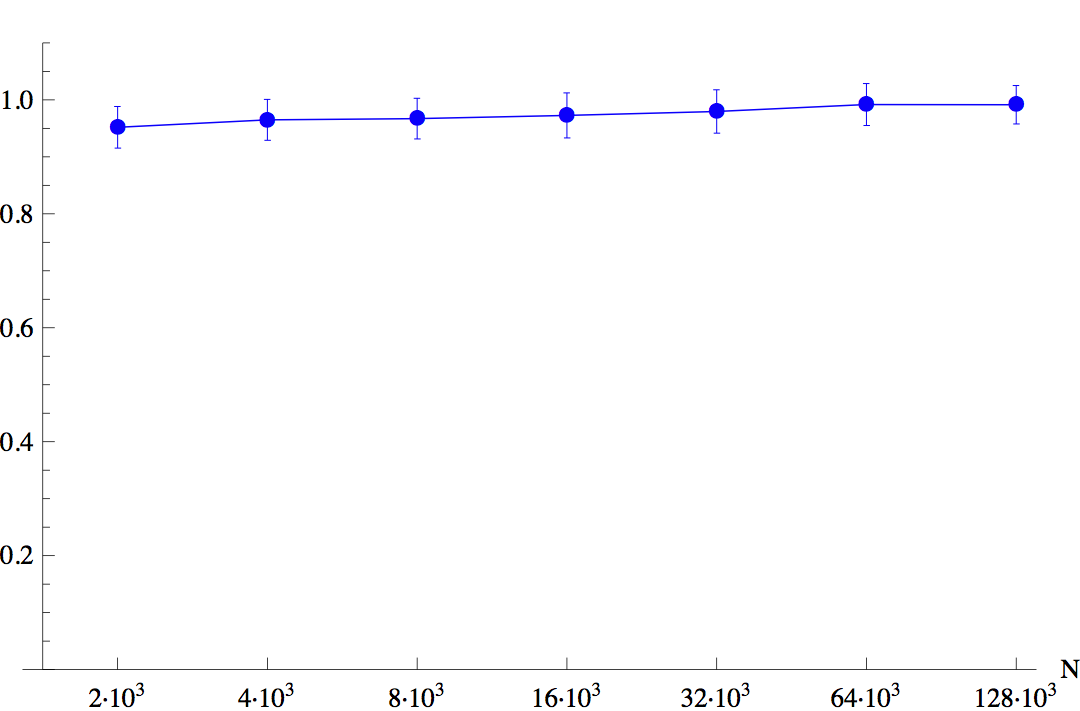}
		\caption{\scriptsize
		$M=8$}\label{figminplotg8}
	\end{subfigure}
	\caption{\small The ratio of $h^{\mathrm{opt}}/h^{\mathrm{opt}}_{\MISE}$
	for different subset configurations.}\label{figminplotg}
\end{figure}

\section{Appendix}


\subsection{Kernel density estimators and asymptotic error analysis}
In this section we will use the following notation. The function $f$ denotes a probability density and its kernel density estimator is given by
\begin{equation}\label{kdensest}
\hat{f}(x; X_1, X_2,\dots, X_N, h )=\frac{1}{N h}\sum_{i=1}^{N}K \left(\frac{x-X_{i}}{h}\right).  \end{equation}
where $X_1,X_2,\dots X_n \sim f$ are i.i.d. samples.

\begin{lemma}[\bf bias expansion]\label{biasformula}
Let $K$ satisfy \eqref{kcond1} and \eqref{kcond2}. Let $f$ be a probability density function satisfying \eqref{pcond1} and \eqref{pcond2}. Let $\fhat_{n,h}(x)$ be an estimation of $f$ given by \eqref{kdensest}. Then
\begin{itemize}
  \item [$(i)$] $\bias(\fhat_{n,h})$ is given by
  \begin{equation}\label{biasexp}
  \begin{aligned}
\big[&\bias (\fhat_{n,h})\big]  (x)= \\
&\;\;\quad = \EXP\big[\fhat_{n,h}(x)\big]-f(x)=\frac{h^2k_2f''(x)}{2}+ {[E_b(f,K)](x \SP ; h)}
\end{aligned}
\end{equation}
where
\begin{equation}\label{biaserr}
  E_b(x;h) :=\int_{\RR}K(t)\Big(\int_{x}^{x-ht}\frac{f'''(z)(x-ht -z)^2}{2}\,dz\Big)\,dt\,.
\end{equation}
  \item [$(ii)$] For all $n\geq 1$ and $h>0$ the term  $E_b(\cdot \SP; n,h)$ satisfies the bounds
  \begin{equation}\label{biaserrbound}
  \begin{aligned}
   & |E_b(x \SP; h)| \, \leq \, \frac{C k_3}{6}h^3\,, \quad x\in \RR \\
    \int_{\RR} & |E_b(x \SP ; h)|\, dx \,\leq  C\frac{k_3}{6} h^3 \\
    \int_{\RR} &|E_b(x \SP ; n,h)|^2\, dx \,\leq \,\frac{C^2 k_3^2}{36} h^6
    \end{aligned}
\end{equation}
	for some constant $C$.
  \item [$(iii)$] The square-integrated $bias(\fhat_{n,k})$ satisfies
  \begin{equation}\label{sqintbias}
    \int_{\RR} \bias ^2(\fhat_{n,k}) \, dx  \, = \, \frac{h^4 k_2^2}{4}\int_{\RR} (f''(x))^2 \, dx  + \mathcal{E}_b(n,h) \, < \, \infty
  \end{equation}
  with
  \begin{equation}\label{intbias2err}
  |\mathcal{E}_b(n,h)| \leq C_b \Big(k_2 + \frac{k_3}{6}h\Big) \frac{k_3 h^5}{6}
  \end{equation}
  for some constant $C_b$, and all $n \geq 1$, $h>0$.
\end{itemize}
\end{lemma}

\begin{proof}
Using \eqref{kdensest} and the fact that $X_i$, $i=1,\dots,n$ are i.i.d. we obtain
\[
\begin{aligned}
\bias_{n,h}(x)&=\EXP\big[\fhat_{n,h}(x)\big]-f(x)=\\
&=\frac{1}{h}\EXP\Big[K\Big(\frac{x-X_1}{h}\Big)\Big]-f(x)\\		&=\frac{1}{h}\int_{\RR}K\Big(\frac{x-y}{h}\Big)f(y)\,dy-f(x)\\
&=\int_{\RR}K(t)\big(f(x-ht)-f(x)\big)\,dt
\end{aligned}
\]
where we used the substitution $t=(x-y)/h$. Employing Taylor's Theorem with an error term in integral form and using \eqref{kcond1} we get
\begin{equation*}
\begin{aligned}
\bias_{n,h}(x)&=\int_{\RR}K(t)\bigg(-htf'(x)+\frac{h^2t^2}{2}f''(x) + \int_{x}^{x-ht}\frac{f'''(z)(x-ht-z)^2}{2}\,dz\bigg)\,dt\\
			&=\frac{h^2f''(x)}{2}\int_{\RR}t^2K(t)\,dt+\int_{\RR}K(t)\Big(\int_{x}^{x-ht}\frac{f'''(z)(x-ht-z)^2}{2}\,dz \Big)\,dt
\end{aligned}
\end{equation*}
which proves $(i)$.

\par

By \eqref{kcond2} we have
\begin{equation}\label{Linferrbiasbnd}
|E_b(x \SP; n,h)|\leqslant C\bigg(\int_{\RR}K(t)\bigg|\int_{x}^{x-ht}\frac{(x-ht -z)^2}{2}\,dz\bigg|\, dt\bigg)= \frac{C k_3 }{6} h^3
\end{equation}
and by \eqref{pcond2}, using the substitution $\alpha=x-ht-z$ and employing Tonelli's Theorem, we obtain

\begin{equation}\label{L1errbiasbnd}
\begin{aligned}
\int_{\RR}  |E_b&(x \SP; n,h)|\,dx \\
& \leq \, \int_{\RR}\int_{\RR}K(t)\int_{x-\frac{h}{2}(|t|+t)}^{x+\frac{h}{2}(|t|-t)}\frac{|f'''(z)|(x-ht -z)^2}{2}\,dz\,dt\,dx\\[2pt]
&= \, \int_{\RR}K(t)\int_{-\frac{h}{2}(|t|+t)}^{\frac{h}{2}(|t|-t)}\bigg(\Big(\int_{\RR}|f'''(x-ht-\alpha)|\,dx\Big)\frac{\alpha^2}{2} \bigg)\,d\alpha\,dt\\
&\leq \, C\int_{\RR}K(t)\bigg(\int_{-\frac{h}{2}(|t|-t)}^{\frac{h}{2}(|t|+t)}\frac{\alpha^2}{2}\,d\alpha\,\bigg)dt =\frac{h^3}{6}C k_3\,.
\end{aligned}
\end{equation}

Thus, combining the two bounds above we conclude
\[
\int_{\RR} |E_b(x \SP; n,h)|^2 \, dx \leq \frac{C k_3 }{6}h^3 \int_{\RR} |E_b(x \SP; n,h)| \, dx \leq \frac{C^2 k_3^2 }{36}h^6\,.
\]

\par

Observe that
\begin{equation}\label{bias2}
\begin{aligned}
\bias^2(\fhat_{n,h})(x) = \frac{h^4k_2^2 }{4}(f''(x))^2 + h^2k_2 f''(x) E_b(x \SP; n,h) + E_b^2(x; n,h)\,.
\end{aligned}
\end{equation}
By \eqref{pcond1}, \eqref{Linferrbiasbnd} and \eqref{L1errbiasbnd}
\begin{equation}\label{L1loworderbnd}
\begin{aligned}
\big|\mathcal{E}_b(n,h)\big|&:=\bigg| \int_{\RR} \Big(h^2k_2 f''(x)E_b(x \SP; n,h) + E_b^2(x; n,h)\Big) \, dx \bigg|\\
& \quad \leq \Big(h^2k_2 C + \frac{C k_3}{6}h^3\Big) \int_{\RR} |E_b(x \SP; n,h)|\\
& \quad \leq \Big(h^2k_2 C + \frac{C k_3}{6}h^3\Big) \frac{h^3}{6}C k_3\,.
\end{aligned}
\end{equation}
By \eqref{pcond1} and \eqref{pcond2} we have $\int_{\RR} (f''(x))^2 \, dx < \infty$. Hence by setting $C_b=C^2$, using \eqref{bias2} and \eqref{L1loworderbnd} we obtain \eqref{intbias2err}.
\end{proof}

\begin{lemma}[\bf variation expansion]\label{variationformula}
Let $K$ satisfy \eqref{kcond1} and \eqref{kcond2}, with $r=2$. Let $f$ satisfy \eqref{pcond1} and \eqref{pcond2}, and $\fhat_{n,h}(x)$ be the estimator of $f$ given by \eqref{kdensest}. Then
\begin{itemize}
  \item [$(i)$] $\VAR(\fhat_{n,h})$ is given by
  \begin{equation}\label{varexp}
\big[\VAR(\fhat_{n,h})\big](x) =f(x)\frac{1}{nh}\int_{\RR}K^2(t)\,dt+ E_{V}(x \SP ; n,h)\,, \quad x \in \RR
\end{equation}
with
\begin{equation}\label{varerr}
\begin{aligned}
  E_V(x;n,h) &= 
    -\frac{1}{n}\bigg(\int_{\RR} t K^2(t) \int_{0}^{1} \SP f'(x-htu)\, du \,dt +\Big(f(x)+\bias(\fhat_{n,h})(x)\Big)^2 \bigg)\,
\end{aligned}
\end{equation}
  \item [(ii)] The term $E_V(x \SP; n,h)$ satisfies
\begin{equation}\label{L1errvarbnd}
\begin{aligned}
\mathcal{E}_V(n,h)&=\left|\int_{\RR} E_V(x)\,dx\right| \\
& \leq \SP \frac{C_V}{n} \bigg( 2 + h^2 k_2 + \big(k_2 + \frac{k_3}{3}h\big) \frac{h^5}{6}k_3\Big)\,.
\end{aligned}
\end{equation}
\end{itemize}
\end{lemma}

\begin{proof} Using \eqref{biasexp} and the fact that $X_i$, $i=1,\dots,n$, are i.i.d. we obtain
\[
\begin{aligned}
\VAR(\fhat_{n,h}(x))&=
\VAR\Big(\frac{1}{h}K\Big(\frac{x-X_1}{h}\Big)\Big)\\
&=\frac{1}{n}\int_{\RR}\frac{1}{h^2}K^2 \left(\frac{x-y}{h}\right)f(y)\,dy-\frac{1}{n}\left(\int_{\RR}\frac{1}{h}K\left(\frac{x-y}{h}\right)f(y)\,dy\right)^2\\
&=\frac{1}{nh}\int_{\RR}K^2(t)f(x-ht)\,dt -\frac{1}{n}\Big(f(x)+\bias(\fhat_{n,h})(x)\Big)^2\\
&=\frac{1}{nh}\int_{\RR}K^2(t)f(x)\,dt +\frac{1}{nh}\int_{\RR}K^2(t)\Big(\int_{x}^{x-ht}f'(z)\,dz\,\Big)dt \\
&\quad -\frac{1}{n}\Big(f(x)+\bias(\fhat_{n,h})(x)\Big)^2\\
&=\frac{1}{nh}\int_{\RR}K^2(t)f(x)\,dt -\frac{1}{n}\int_{\RR}tK^2(t)\int_{0}^{1}f'(x-htu)\,du\,dt \\
&\quad -\frac{1}{n}\Big(f(x)+\bias(\fhat_{n,h})(x)\Big)^2\\
\end{aligned}
\]
which proves \eqref{varexp} and \eqref{varerr}.

\par

We next estimate the terms
\[
\begin{aligned}
E_1(x):=\int_{\RR}t K^2(t)\bigg(\int_{0}^{1}f'(x-htu)\,du\bigg)\,dt\,, \quad E_2(x):= \Big(f(x) + \bias(\fhat_{n,h})(x)\Big)^2\,.
\end{aligned}
\]
Observe that \eqref{pcond1}-\eqref{pcond2} imply
\[
\int_{\RR} |f'(x)|\, dx = \int_{\RR} |f'(x+\alpha)|\, dx :=I_1\, < \, \infty
\]
for any $\alpha \in \RR$. Then using Tonelli's Theorem ans \eqref{kcond2} we obtain
\[
\begin{aligned}
\int_{\RR} |E_1(x)|\, dx & \leq \int_{\RR} |t|K^2(t)\bigg(\int_{\RR} \int_{0}^{1} |f'(x-htu)|\,du\,dx \bigg)\, dt\\
& \leq \int_{\RR} |t|K^2(t) \bigg(\int_{0}^{1} \Big(\int_{\RR} |f'(x-htu)|\,dx \Big)\,du \bigg)\, dt  \leq I_1  k_1\\
\end{aligned}
\]
Since $E_1$ is integrable we can use Fubini's Theorem and this yields
\[
\begin{aligned}
\int_{\RR} E_1(x)\, dx & = \int_{\RR} t K^2(t)\bigg(\int_{\RR} \int_{0}^{1} f'(x-htu)\,du\,dx \bigg)\, dt\\
& = \int_{\RR} tK^2(t) \bigg(\int_{0}^{1} \Big(\int_{\RR} f'(x-htu)\,dx \Big)\,d u \bigg)\, dt = 0
\end{aligned}
\]
where we used the fact that $\lim_{x\to \pm \infty}f(x) = 0$. Next, by \eqref{pcond1} and \eqref{biaserrbound} we get
\[
\begin{aligned}
\int_{\RR} |E_2(x)| \, dx & \leq 2 \int_{\RR} \Big(f^2(x) + \bias^2(\fhat_{n,h})(x) \Big)\, dx\\ 
&\leq 2 C + C h^2 k_2  + \Big(k_2 C + \frac{C k_3}{6}h\Big) \frac{h^5}{3}C k_3\,.
\end{aligned}
\]
Combining the above estimates we obtain \eqref{L1errvarbnd}\,.
\end{proof}
\begin{lemma}[\textbf{kernel autocorrelation}]\label{convoLemma}
Let $K$ satisfy \eqref{kcond1} and \eqref{kcond2}, then the function
\[
	K_2(z)=\int_{\RR}K(s)K(s-z)\,ds \geq 0\,, \quad z \in \RR
\]
satisfies
\[
\begin{aligned}
	\int_{\RR}K_2(z)\,dz&=1, \quad   \int_{\RR}z\,K_2(z)\,dz&=0\,.
\end{aligned}
\]
Moreover, for any sufficiently smooth $f(x)$
\[
    \int \frac{1}{h}K_2\left(\frac{z-x}{h}\right)f(z)\,dz=f(x) +E_{C,f} \quad \text{with} \quad |E_{C,f}|\leq  \|f''\|_{\infty}k_2h^2\,.
\]
\end{lemma}
\begin{proof}
Since $K \geq 0$ we have $K_2 \geq 0$. Moreover, we have
\[
	\int_{\RR}K_2(z)\,dz=\iint_{\RR\times\RR}K(s)K(s-z)\,dzds=1
\]
and this proves the first property. Similarly, recalling that $\int z K(z) \,dz =0$, we obtain
\[
\begin{aligned}
    \int_{\RR}z\,K_2(z)\,dz& 
    =\int_\RR K(s)\int_\RR(z-s+s)K(s-z)\,dz\,ds=0\,.
\end{aligned}
\]

\par

Next, we take any smooth function $f$ and compute
\[
\begin{aligned}
	&\int \frac{1}{h}K_2\left(\frac{z-x}{h}\right)f(z)\,dz 
    =\int K_2\left(u\right)f(x-hu)\,du\\
    &\qquad=f(x)+\int K_2(u)\int_{x}^{x-hu}f''(t)(t-x+hu)\,dt\,du\,.
\end{aligned}
\]

\par

Finally, we estimate the last term in the above formula as follows
\[
\begin{aligned}
	&\left|\int K_2(u)\int_{x}^{x-hu}f''(t)(t-x+hu)\,dt\,du\right|\\
    &\qquad\leq  \|f''\|_{\infty}\int K_2(u)\frac{h^2u^2}{2}\,du\\
    &\qquad= \frac{ \|f''\|_{\infty}h^2}{2}\left(\int K(s)\int (s-u)^2K(s-u)\,duds+\int s^2K(s)\int K(s-u)\,duds\right)\\
    &\qquad\leq \|f''\|_{\infty}k_2h^2\,.
\end{aligned}
\]
\end{proof}

\begin{lemma}[\textbf{product expectation}]\label{expProdLemma}
Let $K$ satisfy \eqref{kcond1} and \eqref{kcond2}, with $r=2$. Let $f$ be a probability density function that satisfies \eqref{pcond1} and \eqref{pcond2}, and let $\fhat_{n,h}(x)$ be an estimate of $f$ given by \eqref{kdensest}. Then
\begin{equation}\label{covexpr}
  \EXP[\fhat_{n,h}(x)\fhat_{n,h}(y)]-\EXP[\fhat_{n,h}(x)]\EXP[\fhat_{n,h}(y)]=\frac{1}{Nh}f(x) K_2\Big(\frac{x-y}{h} \Big)-E_{\Pi},
\end{equation}
where the error term
\[
E_{\Pi}=\frac{1}{N} \int \bigg(\, s K (s)  K\big(s - \frac{x-y}{h} \big) \Big(\int_{0}^{1}f'(x-shu)\,du \big)\bigg) ds+\frac{1}{N}\EXP[\fhat(x)]\EXP[\fhat(y)]
\]
satisfies
\begin{equation}\label{estcov1d}
\begin{aligned}
&|E_{\Pi}(x,y)|\leq \frac{C_\Pi}{N}\,, \qquad \Big|\int\int E_{\Pi}(x,y) \, dx dy\Big| 
\leq \frac{1}{N}\left(1+\frac{Ck_3h^3}{6}\right)^2\\
 &\int \int\Big| E_{\Pi}(x,y)\Big|dx dy  \leq \frac{1}{N}\left(1+k_1\,C\frac{C k_2h^2}{2}+\frac{Ck_3h^3}{6}\right)^2
\end{aligned}
\end{equation}
for some constant $C_\Pi$ and constants $C$ given in \eqref{pcond2} and $K_2$  defined in Lemma \ref{convoLemma}.


\begin{proof}
By the definition of the estimator $\fhat$ we have
\begin{equation}
  \begin{aligned}
   \EXP \Big(\fhat(x) \fhat(y) \bigg) 
   &=\EXP \bigg(\frac{1}{N^2h^2}\sum_{i,j=1}^N K\Big(\frac{x-X_i}{h}\Big)K\Big(\frac{y-X_j}{h}\Big) \bigg)\,.
  \end{aligned}
\end{equation}
Since all $\{X_i\}_{i=1}^{N}$ are i.i.d. we can split the calculation into two parts, one for the part, where the indexes coincide and the part, where indexes are different. We then can use the independence of the samples to simplify the calculation
\begin{equation}\label{expprodmixvar}
  \begin{aligned}
   \EXP \Big(\fhat(x) \fhat(y) \bigg)&=\frac{1}{N^2h^2} \EXP \bigg(\sum_{i=j} K(\frac{x-X_i}{h})K(\frac{y-X_i}{h}) \bigg)\\
   &\qquad+\frac{1}{N^2h^2} \EXP \bigg( \sum_{i \neq j}  K\Big(\frac{x-X_i}{h}\Big)K\Big(\frac{y-X_j}{h}\Big)\bigg)\\
   &{ = \frac{1}{Nh^2}\left[ \EXP \Big( K\Big(\frac{x-X}{h}\Big)K\Big(\frac{y-X}{h} \Big) \bigg)\right] + \Big(1-\frac{1}{N}\Big)\EXP[\fhat(x)]\EXP[\fhat(y)]}\\
  \end{aligned}
\end{equation}
where $X=X_1$. The first expectation term in \eqref{expprodmixvar} can be expanded as
\[
\begin{aligned}
  &\frac{1}{Nh^2}\EXP \left[ K\left(\frac{x-X}{h}\right)K\left(\frac{y-X}{h} \right) \right] \\
  & \qquad = \frac{1}{Nh^2} \int \,  K \Big(\frac{x-t}{h}\Big) K\Big(\frac{y-t}{h} \Big) \, f(t) \, dt\\
     &\qquad=\frac{1}{Nh} \int \,  K (s)  K\Big(s - \frac{x-y}{h} \Big) \, \Big( f(x)+\int_{x}^{x-sh}f'(z)\,dz \Big) \, ds\\
          &\qquad= f(x)\frac{1}{Nh} K_2\Big(\frac{x-y}{h} \Big)\\
     &\qquad\qquad-\frac{1}{N} \int \, s K (s)  K\Big(s - \frac{x-y}{h} \Big) \, \bigg(\int_{0}^{1}f'\left(x-shu\right)\,du \bigg) \, ds\\
\end{aligned}
\]
Let us denote
\[
E_{\Pi,1}=\frac{1}{N} \int \bigg(\, s K (s)  K\big(s - \frac{x-y}{h} \big) \Big(\int_{0}^{1}f'(x-shu)\,du \big)\bigg) ds\,, \quad E_{\Pi,2}=\frac{1}{N}\EXP[\fhat(x)]\EXP[\fhat(y)].
\]

Then we obtain
\begin{equation*} 
  \begin{aligned}
   &\EXP \Big(\fhat_{n,h}(x) \fhat_{n,h}(y) \bigg) - \EXP[\fhat_{n,h}(x)]\EXP[\fhat_{n,h}(y)]\\
   &\qquad = f(x)\frac{1}{Nh} K_2 \Big(\frac{x-y}{h} \Big)  ds-(E_{\Pi,1}+E_{\Pi,2}).\\
  \end{aligned}
\end{equation*}
and this establishes \eqref{covexpr}.
\par

Observe that \eqref{kcond1}, \eqref{kcond2} and \eqref{pcond1} imply
\[
|E_{\Pi,1}|\leq \frac{C\,k_1}{N}\,.
\]
Next, according to \eqref{biasexp} and \eqref{biaserrbound}
\[
|\EXP[\fhat(x)]|\leq C+\frac{C k_2 h^2}{2}+\frac{C k_3 h^3}{6} \quad \text{for all} \quad x\in \RR
\]
where $C$ is a maximum of constants from \eqref{pcond1} and hence
\[
|E_{\Pi,2}| \leq \frac{1}{N}\Big( C+\frac{C k_2 h^2}{2}+\frac{C k_3 h^3}{6} \Big)^2\,.
\]

Combining the above estimate
we conclude that
\[
|E_{\Pi}|=|E_{\Pi,1}+E_{\Pi,2}|   \leq \frac{1}{N}\left(Ck_1+\left(C+\frac{C k_2 h^2}{2}+\frac{C k_3 h^3}{6}\right)^2\right).
\]
To obtain bounds on the integral of the error term, let us consider each component of the error separately. The term $E_{\Pi,1}$ is integrable
\begin{equation}\label{coverrL1norm1}
\begin{aligned}
  \iint |E_{\Pi,1}(x,y)|\,dxdy &\leq \frac{1}{N} \iiint_{\RR^3} \, |s|\, K (s)  K\Big(s - \frac{x-y}{h} \Big) \, \Big(\int_{0}^{1}|f'\left(x-shu\right)|\,du \Big) ds\,dx\,dy\\
      &\leq \frac{1}{N} \int_{\RR} \, |s|\, K (s) \, \Big(\int_{0}^{1} \int_\RR|f'\left(x-shu\right)|\,dx \,du \Big)  ds\leq\frac{k_1\,C}{N}
\end{aligned}
\end{equation}
Next using Fubini Theorem, we obtain
\[
\begin{aligned}
  &\left|\iint E_{\Pi,1}(x,y)\,dxdy\right| \\
  & \leq \frac{1}{N} \left|\iiint_{\RR^3} \, s\, K (s)  K\Big(s - \frac{x-y}{h} \Big) \, \Big(\int_{0}^{1}f'\left(x-shu\right)\,du \Big) ds\, dx\,dy\right|\\
      &= \frac{1}{N} \left|\int_{\RR} \, s\, K (s) \, \Big( \int_{0}^{1}\int_\RR f'\left(x-shu\right)\,dx\,du \Big)\, ds\right|=0\,.
\end{aligned}
\]
Therefore, using Lemma \ref{biasformula}, \eqref{biasexp}, \eqref{biaserrbound} and the hypothesis \eqref{pcond2} we obtain

\[
\begin{aligned}
  \left|\iint_{\RR^2} E_{\Pi}(x,y)\,dxdy\right| & 
    =\frac{1}{N} \left|\int_{\RR} \EXP[\fhat(x)]\,dx\right|^2 
    \leq    \frac{1}{N}\left(1+\frac{Ck_3h^3}{6}\right)^2\,.
\end{aligned}
\]


Finally,  directly from \eqref{coverrL1norm1}, \eqref{biasexp} and \eqref{biaserrbound} we obtained
\[
\begin{aligned}
  \iint_{\RR^2} \left|E_{\Pi}(x,y)\right| \,dxdy &\leq \iint_{\RR^2} \left|E_{\Pi,1}(x,y)\right| dxdy+ \iint_{\RR^2} \left|E_{\Pi,2}(x,y)\right| dxdy\\
  &\leq \frac{k_1\,C}{N} + \frac{1}{N}\left(1+\frac{Ck_2h^2}{2}+\frac{Ck_3h^3}{6}\right)^2 \\
\end{aligned}
\]

\end{proof}
\end{lemma}


\begin{theorem}[\bf MISE expansion]\label{miseest}
Let $K$ satisfy \eqref{kcond1} and \eqref{kcond2}, with $r=2$. Let $f$ be a probability density function that satisfies \eqref{pcond1} and \eqref{pcond2}, and let $\fhat_{n,h}(x)$ be an estimate of $f$ given by \eqref{kdensest}. Then
\begin{equation}\label{miseexact}
\MISE(\widehat{f}_{n,h})
=\frac{h^4k^2_2}{4}\int_{\RR}(f''(x))^2dx
+ \frac{1}{nh} f(x)\int_{\RR} K^2(t) \,dt
+ \mathcal{E}_b(n,h)+\mathcal{E}_V(n,h)
\end{equation}
with $\mathcal{E}_b$ and  $\mathcal{E}_V$ defined in \eqref{L1loworderbnd} and \eqref{L1errvarbnd}, respectively. Moreover, for every $H>0$ there exists $C_{f,K,H}$ such that
\begin{equation}\label{L1errmisebnd}
|\mathcal{E}_b(h,n)+\mathcal{E}_V(h,n)| \, \leqslant \, C_{f,K,H} \Big( h^5 + \frac{1}{n}\Big)
\end{equation}
for all $n \geq 1$ and  $H \geq  h > 0$.
\end{theorem}

\begin{proof}
It is easy to show (see \cite{SILV85}) that
\[
\begin{aligned}
\MISE(\fhat_{n,h})&=\int_{\mathbb{R}}\EXP[\fhat_{n,h}(x)-f(x)]^2\,dx\\
& = \int_{\mathbb{R}}\big(\bias(\fhat_{n,h})(x)\big)^2\, dx \SP +\int_{\mathbb{R}}\VAR(\fhat_{n,h}(x))\,dx\,.
\end{aligned}
\]
and hence the result follows from Lemma \ref{biasformula} and Lemma \ref{variationformula}.
\end{proof}

\end{document}